\documentclass[11pt,a4paper,leqno]{amsart}
\usepackage[T1]{fontenc}
\usepackage[utf8]{inputenc}
\usepackage{lmodern,amsmath,amssymb,amsthm,mathtools}
\usepackage[margin=28mm]{geometry}
\usepackage[expansion=false]{microtype}
\usepackage{needspace}
\usepackage{xcolor}
\definecolor{linkblue}{RGB}{28,54,84}
\usepackage[colorlinks=true,linkcolor=linkblue,citecolor=linkblue,urlcolor=linkblue]{hyperref}
\hypersetup{pdftitle={A sharp threshold for mixed Q-curvature rigidity},pdfauthor={Wangzhe Wu}}
\numberwithin{equation}{section}
\newtheorem{theorem}{Theorem}[section]
\newtheorem{proposition}[theorem]{Proposition}
\newtheorem{lemma}[theorem]{Lemma}
\newtheorem{corollary}[theorem]{Corollary}
\newtheorem{maintheorem}{Theorem}

\theoremstyle{definition}
\newtheorem*{remark}{Remark}
\newcommand{\Ric}{\operatorname{Ric}}
\newcommand{\tr}{\operatorname{tr}}
\newcommand{\Vol}{\operatorname{Vol}}
\newcommand{\R}{\mathbb R}
\newcommand{\dd}{\mathrm d}
\allowdisplaybreaks[3]
\newcommand{\sech}{\operatorname{sech}}
\newcommand{\Sph}{\mathbb S}
\newcommand{\cK}{\mathcal K}
\title[A sharp threshold for mixed curvature rigidity]
{A sharp threshold for mixed $Q$-curvature rigidity}
\author{Wangzhe Wu}
\address[Wangzhe Wu]{School of Mathematics and Statistics,
Ningbo University, Ningbo, Zhejiang, People's Republic of China}
\email{wuwz18@mail.ustc.edu.cn}
\date{}
\subjclass[2010]{Primary 53C21; Secondary 35J60, 34C37, 53C24}
\keywords{Conformal geometry, Q-curvature, sigma-two curvature, Einstein metric,
sharp rigidity, nonuniqueness, neck construction}
\begin{document}

\begin{abstract}
We resolve Case's question (Crelle's Journal, 2024) on the sharp lower
threshold for Obata-type rigidity of $I_a=Q+a\sigma_2(A)$
in positive Einstein conformal classes of dimension $n\ge4$.
Using a new reference-curvature identity, we prove that,
on a closed connected manifold, every smooth metric in such
a class with nonnegative scalar curvature and constant $I_a$
is Einstein whenever $a\ge-4$.
The scalar-curvature sign assumption can be removed for
$-4\le a\le-2(n-2)/(n-1)$.
We also extend the quotient rigidity theorem of
Ge--Wang--Wei~\cite[Theorem~3.2]{GWWQuot26}
to $I_a=\Lambda R^\theta$ for $R>0$, $a\ge-4$,
and $\theta\le1$.
Sharpness is established in every dimension by a shooting
construction. For each sufficiently small $\eta>0$, we
construct a smooth non-Einstein metric in the round conformal
class on $\Sph^n$ with positive scalar curvature and the same
constant $I_{-4-\eta}$ as the unit round metric.
The construction smoothly matches a perturbed Schwarzschild
neck near the equator with perturbed round caps at both poles.

\end{abstract}
\maketitle

\section{Introduction}
The Yamabe theorem~\cite{Aubin76,Schoen84,Trudinger68,Yamabe60}
provides a constant-scalar-curvature metric in every conformal
class on a closed manifold of dimension at least three;
see~\cite{LeeParker87}. When the class contains a positive
Einstein metric, Obata's theorem~\cite{Obata71} asserts that
every such metric is Einstein. It is unique up to scaling,
apart from conformal diffeomorphisms of the round sphere.

Fourth-order analogues involve Branson's $Q$-curvature.
Chang--Yang~\cite{CY95} studied its role in determinant
functionals, and Djadli--Malchiodi~\cite{DM08} obtained
constant-$Q$ metrics on closed four-manifolds under generic
assumptions. V\'etois~\cite{Vetois23} proved constant-$Q$
rigidity in positive Einstein conformal classes for $n\ge4$.
The related $\sigma_2$-curvature problem includes the construction
of positive-Ricci-curvature metrics by Chang--Gursky--Yang~\cite{CGY02}.
For $n\ge5$, Ge--Wang--Wei~\cite[Theorem~1.1]{GWWSigma26}
prove attainment of the $\sigma_2$-Yamabe minimum over metrics with
positive scalar curvature, assuming positivity of the first two
Yamabe constants. Their sphere examples show why this scalar-curvature
restriction matters for that variational problem~\cite[Section~6]{GWWSigma26}.
They also establish sharp geometric inequalities for total
$\sigma_k$-curvature in a larger admissible cone; for $k=2$,
the cone requires only positive scalar curvature~\cite[Theorem~1.1]{GWWFlow26}.

We consider the mixed curvature $I_a(g):=Q_g+a\sigma_2(A_g)$.
The variational properties of such combinations under conformal
changes were studied by Case--Lin--Yuan~\cite[Section~8.3]{CLY19}.
Varying $a$ changes the coefficients of the quadratic Schouten terms
while retaining the fourth-order part of $Q_g$. This leads to the
question of how $a$ affects rigidity. Case~\cite{Case24} extended
the Obata--V\'etois argument to this family. Under nonnegative scalar curvature,
he obtained rigidity for a finite interval of $a$ containing
zero and asked whether this range is optimal; see the discussion
following~\cite[Theorem~1.1]{Case24}. Li--Wei~\cite{LiWei25}
enlarged this interval, but the optimal range remained open.
Gursky--Malchiodi~\cite{GM12} constructed
non-Einstein examples for certain values of $a$ in the round
conformal class on $\Sph^4$, without asserting a sign for their
scalar curvature. The remaining issue is the sharp lower threshold under
nonnegative scalar curvature, with sharpness tested by examples
that retain strictly positive scalar curvature.

We resolve Case's question in dimensions $n\ge4$ by determining
the sharp lower threshold in positive Einstein conformal classes.
In the conformal class of a positive Einstein
metric on a closed manifold, a metric with nonnegative scalar
curvature and constant $I_a(g)$ is Einstein whenever $a\ge-4$.
For every $a$ sufficiently close to $-4$ from below, we construct non-Einstein
conformal metrics on the round sphere with strictly positive
scalar curvature and the same constant $I_a$ as the unit round
metric. These examples show that the half-line $[-4,\infty)$ cannot be extended to the left.
They also describe how rigidity fails: a neck develops,
and its rescaled limit is the Riemannian Schwarzschild metric.

To state the results, let $\Ric_g$ and $R_g=\tr_g\Ric_g$ denote
the Ricci tensor and scalar curvature. The Schouten tensor and
its first two elementary symmetric curvatures are
\begin{equation}\label{intro:schouten-defs}
 \begin{gathered}
 A_g\coloneqq\frac1{n-2}\left(\Ric_g-\frac{R_g}{2(n-1)}g\right),
 \qquad \sigma_1(A_g)\coloneqq\tr_gA_g=\frac{R_g}{2(n-1)},\\
 \sigma_2(A_g)\coloneqq\frac12\left(\sigma_1(A_g)^2-|A_g|_g^2\right).
 \end{gathered}
\end{equation}
All traces and norms are taken with respect to $g$. With
$\Delta_g f=\tr_g(\nabla_g^2f)$, we use Branson's normalization
\begin{equation}\label{intro:Q}
 Q_g\coloneqq-\Delta_g\sigma_1(A_g)-2|A_g|_g^2
       +\frac n2\sigma_1(A_g)^2,
 \qquad I_a(g)\coloneqq Q_g+a\sigma_2(A_g),\qquad a\in\R.
\end{equation}
These conventions agree with Li--Wei~\cite{LiWei25}.
A manifold is closed if it is compact without boundary.

\Needspace*{8\baselineskip}
\begin{maintheorem}[Rigidity]\label{main:rigidity}
Let $(M^n,g_0)$ be closed and connected, $n\ge4$, with $g_0$ Einstein
and $R_{g_0}>0$. Suppose $g=u^2g_0$, where $u$ is smooth and positive,
satisfies
\[
 R_g\ge0,\qquad I_a(g)=\Lambda,\qquad a\ge-4,
\]
for a constant $\Lambda$. Then $g$ is Einstein and $R_g>0$.
Furthermore, $g=s^2g_0$ for a constant $s>0$, unless $(M,g_0)$
is a round sphere up to scale. In the round case,
$g=s^2\Phi^*g_0$ for a conformal diffeomorphism $\Phi$.
In either case,
\begin{equation}\label{main:rigidity-value}
 \Lambda=\frac{(n-1)a+n^2-4}{8n(n-1)^2}R_g^2.
\end{equation}
\end{maintheorem}

The scalar-curvature sign assumption can be removed when
\[
 -4\le a\le-\frac{2(n-2)}{n-1}.
\]
Under the same assumptions on $(M,g_0)$, every smooth metric conformal
to $g_0$ with constant $I_a(g)$ in this range has positive scalar
curvature and is Einstein; see Proposition~\ref{four:unsigned}.

The lower endpoint is sharp even among metrics with strictly positive
scalar curvature.

\begin{maintheorem}[Sharpness]\label{main:sharpness}
For each integer $n\ge4$, there exists $\eta_n>0$ such that, for every
$0<\eta<\eta_n$, there is a smooth metric $g_{n,\eta}$ on $\Sph^n$,
conformal to its unit round metric, satisfying
\[
 R_{g_{n,\eta}}>0,\qquad g_{n,\eta}\text{ is not Einstein},
\]
and
\begin{equation}\label{sharp:normalization}
 I_{-4-\eta}(g_{n,\eta})
 =I_{-4-\eta}(g_{\Sph^n}).
\end{equation}
These metrics are rotationally symmetric and invariant under reflection
across the equator. Here rotational symmetry means invariance under
rotations fixing two antipodal points.
\end{maintheorem}

\subsection{The rigidity mechanism and the geometry of the examples}

The rigidity argument combines the pointwise Obata identity with
a Newton identity containing a constant reference curvature.
The resulting identity, Proposition~\ref{id:mixed}, holds for every
smooth conformal metric $g=u^2g_0$ with $g_0$ Einstein, without a
curvature equation, compactness, or curvature-sign assumptions.
It avoids the mixed contraction $E_g(\nabla_g R_g,\nabla_g u)$ in the
Case--Gursky--V\'etois identity~\cite{Case24,LiWei25}, where $E_g$
is the trace-free Ricci tensor.
For Theorem~\ref{main:rigidity}, we choose the reference curvature
to be $\min_M R_g$. Evaluating the curvature equation at a point
attaining this minimum makes every term in the integrated identity
nonnegative for all $a\ge-4$, including the endpoint.

The construction in Theorem~\ref{main:sharpness} joins perturbed
round caps near the poles to a neck near the equator modeled on the
scalar-flat Schwarzschild metric~\cite[Section~1]{BrayLee09}.
Lemma~\ref{neck:variation} identifies this metric as the smooth
local limit after rescaling by the inverse square of the neck radius.
The cap convergence in the completion of the proof gives the round
models near the two poles. These limits explain the geometric
mechanism by which rigidity fails below $a=-4$.

\subsection{Earlier rigidity results}
We now compare the parameter ranges in the results above.
For fourth-order curvature equations, Gursky~\cite{Gursky97}
used integral identities to classify critical metrics of certain
determinant functionals in the round conformal class on $\Sph^4$.
V\'etois~\cite{Vetois23} established constant-$Q$ rigidity on
positive Einstein conformal classes without assuming local
conformal flatness. Case~\cite{Case24} extended the integral
argument to $I_a$, and Li--Wei~\cite{LiWei25} sharpened the
estimate controlling its mixed term. Under $R_g\ge0$, their rigidity interval is
$[a_-(n),a_+(n)]$~\cite[Theorem~1.2]{LiWei25}, where
\[
 a_\pm(n)\coloneqq
 \frac{(n-2)(n-4)\pm n\sqrt{n^2+4n}}{2(n-1)}.
\]
For $n\ge4$,
\[
 a_-(n)+4=\frac{2n}{(n-1)(n+2+\sqrt{n^2+4n})}>0.
\]
Theorem~\ref{main:rigidity} extends the conclusion to $a\ge-4$.
In dimension four, without a sign assumption on $R_g$, Case's interval
$[-(2+2\sqrt{21})/3,-4/3]$ was extended by Li--Wei to
$[-8\sqrt2/3,-4/3]$; see~\cite[Theorem~1.2]{Case24} and
\cite[Corollary~1.6]{LiWei25}.
Proposition~\ref{four:unsigned} reaches the lower endpoint $-4$
in every dimension $n\ge4$. It combines Li--Wei's weighted
identity~\cite[(3.1)]{LiWei25} with Case's
conformal-Laplacian positivity argument~\cite[Lemma~4.3]{Case24}
and Theorem~\ref{main:rigidity}. In dimension four, the resulting
interval is $[-4,-4/3]$.

For the quotient equation, Ge--Wang--Wei~\cite[Theorem~3.2]{GWWQuot26}
prove rigidity for $Q_g=\Lambda R_g^\theta$, $0<\theta\le1$, under
$R_g>0$ in dimensions $n\ge3$. In dimensions $n\ge4$,
Corollary~\ref{rig:quotient} gives a new proof of their rigidity
theorem using the reference-curvature identity~\eqref{id:mixed-formula}.
It also extends the conclusion to $I_a(g)=\Lambda R_g^\theta$
for all $a\ge-4$ and $\theta\le1$, with $\Lambda\in\R$.
In each of these results, the background metric is Einstein
with positive scalar curvature.

The Einstein assumption on $g_0$ in Corollary~\ref{rig:quotient}
and in the quotient rigidity theorem of
Ge--Wang--Wei~\cite[Theorem~3.2]{GWWQuot26} cannot be dropped,
even for constant $Q/R$.
For $n\ge26$, Li--Wang--Wei~\cite[Theorem~1.1]{CLWW26} construct
a smooth background metric on $\Sph^n$ that is not conformally flat,
together with a sequence of smooth conformal metrics having positive
scalar curvature and the same positive constant value of $Q/R$.
Their conformal factors are unbounded. These metrics cannot all
be Einstein: Obata's theorem~\cite{Obata71} would make them
homothetic, and the fixed positive quotient would then make them
identical, contradicting the blow-up. A non-Einstein member of the
sequence therefore gives a counterexample if the Einstein
assumption on the background is omitted.

\subsection{The nonuniqueness construction}
Gursky--Malchiodi~\cite[Theorem~1.3]{GM12} constructed smooth nonround
critical metrics in the round conformal class on $\Sph^4$ for the
determinant of the Paneitz operator and for Cheeger's half-torsion.
Their parameter is $\beta=\gamma_2/(12\gamma_3)$, where
$\gamma_2$ and $\gamma_3\ne0$ are the coefficients of their
$Q$-curvature and of $-\Delta_gR_g$, respectively, in the
Euler--Lagrange curvature. Their normalization of $Q$-curvature is half that in
\eqref{intro:Q}. By~\eqref{intro:schouten-defs} and~\eqref{intro:Q},
$Q_g=-\frac16\Delta_gR_g+4\sigma_2(A_g)$ in dimension four, so their equation is equivalent, for $\beta\ne-1$,
to constant $I_a(g)$ with
\[
 a=-\frac4{1+\beta}.
\]
The Paneitz determinant and half-torsion correspond, respectively,
to $\beta=-7/16$ and $\beta=-31/58$, hence to $a=-64/9$ and
$a=-232/27$.

In \cite[Section~5.2]{GM12}, the authors discuss an extension to
general coefficients in the range
\[
 -1<\beta<-\frac14,\qquad\text{equivalently}\qquad a<-\frac{16}{3}.
\]
This range stays a positive distance below $-4$; the uniform estimates
in their construction break down for $\beta>-1/4$.
Moreover, their nonuniqueness theorem makes no assertion about the
sign of $R_g$.  Our Theorem~\ref{main:sharpness} provides examples with
$R_g>0$ for every sufficiently small $\eta>0$ at $a=-4-\eta$, in
every dimension $n\ge4$.

In dimension four, the constructions also differ in their central
initial data. Write
$g=e^{2v(t)}(\dd t^2+g_{\Sph^3})$, and set $x=-v'$ and $y=x'$.
Reflection symmetry about $t=0$ makes $v$ even and gives $x(0)=0$.
The data in \cite[(5.11)]{GM12} satisfy $y(0)>0$, so the radius
$e^{v(t)}$ has a strict local maximum at the center. Our data have
$y(0)<0$, which gives a strict local minimum and produces a neck.

To prove Theorem~\ref{main:sharpness}, we solve the radial ODE
starting from a symmetric neck at $t=0$. We work only on $t\ge0$,
that is, from the equator toward the south pole $P_S$, and aim at
smooth closure as $t\to+\infty$. Reflection then gives the
north-pole half of the metric.

We compare initial data on a common section of the ODE phase
space. Section~\ref{sec:radial} constructs a curve of initial
data whose trajectories give metrics smooth at the south pole.
The matching condition is that the neck trajectory meets this
south-pole curve. Reflecting the construction gives a north-pole curve
on the same section, which we use for comparison.

The main difficulty is that the travel time from the neck to
this section tends to infinity as $\eta\downarrow0$.
Lemma~\ref{neck:variation} first determines the leading scalar
curvature near the neck. Lemma~\ref{passage} follows the trajectories
to their first upward return to the common section and proves
that their return points converge to the reflected north-pole curve.
Its proof uses the squared radius to reach a fixed comparison radius,
then continues over a bounded time interval.
An energy identity for the scalar linearized equation shows that
the north- and south-pole curves have opposite slopes.
The resulting sign change persists under perturbation, so the
intermediate value theorem gives a match for small $\eta>0$
(Lemma~\ref{shooting}).

Finally, we check positive scalar curvature along the resulting
metric. The neck makes it non-Einstein, and constant rescaling
gives the normalization~\eqref{sharp:normalization}.

\subsection{Related methods}
The classification of conformal equations also includes the
Euclidean Liouville theorems of Lin~\cite{Lin98} and
Wei--Xu~\cite{WeiXu99}. Differential identities have been used
for subcritical biharmonic equations on complete manifolds with
nonnegative Ricci curvature by Ma--Wu--Wu~\cite{MWW25}, and for
the divergence structure and weak continuity of $\sigma_k$-curvature
in the locally conformally flat admissible setting by Ma--Wu~\cite{MaWu25}.
In dimensions $n\ge5$, the strong maximum principle
of Gursky--Malchiodi~\cite{GM15} under $R_g\ge0$, $Q_g\ge0$,
and $Q_g\not\equiv0$ provides a related approach to positivity.
The companion paper~\cite{CWSigma} develops boundary rigidity,
sharp $\sigma_2$ inequalities, and stability using a Newton
identity with variable reference curvature. The identities and
arguments needed here are proved in full below.

\paragraph{\textbf{Organization of the paper.}}
Section~\ref{sec:identities} introduces the Newton identity with a
reference curvature and the pointwise mixed-curvature identity.
Section~\ref{sec:rigidity} proves the rigidity theorem and its
extensions. Section~\ref{sec:radial} derives the radial equations
and constructs local solutions whose metrics extend smoothly across a pole.
Section~\ref{sec:high-neck} proves Theorem~\ref{main:sharpness}
by matching the neck solutions with the smooth pole caps.
\section{Reference curvature identities}\label{sec:identities}

\subsection{Notation and conventions}
Let $(M^n,g_0)$ be Einstein, $n\ge4$, and let $g=u^2g_0$, where
$u>0$ is smooth. The local identities require neither compactness nor a
curvature-sign assumption. Whenever we integrate an identity,
we assume that $M$ is closed.
We denote the Levi--Civita connection of $g$ by $\nabla_g$.
For a function $f$, $\nabla_g f$ and $\nabla_g^2f$ denote its gradient
and Hessian. We retain the metric subscript in $\nabla_g$,
$\Delta_g$, and $\operatorname{div}_g$ throughout.
Unless otherwise stated, contractions and norms are taken with
respect to $g$.
Our curvature conventions follow Li--Wei~\cite[Section~1]{LiWei25}.

For a symmetric covariant two-tensor $T$, write
$T^\circ=T-(\tr_gT)g/n$. Its action on vectors is understood
through the endomorphism $g^{-1}T$. In particular,
\begin{equation}\label{id:notation}
 E_g\coloneqq\Ric_g-\frac{R_g}{n}g=(n-2)A_g^\circ,\qquad
 \sigma_1(A_g)=\frac{R_g}{2(n-1)}.
\end{equation}
We use the Laplace--Beltrami operator
$\Delta_g f:=\tr_g(\nabla_g^2f)$ and define
$(\operatorname{div}_gT)_j:=(\nabla_g)^iT_{ij}$.
Indices are raised using $g$, and repeated indices are summed.
Thus $-\Delta_g$ has nonnegative spectrum on a closed manifold.
We identify one-forms and vector fields using $g$.

The volume form is $\dd v_g$, and
$\Vol_g(M)=\int_M\dd v_g$. The symbol $|\Sph^n|$ denotes the
volume of the unit round $n$-sphere. A constant rescaling $s^2g$,
with $s>0$, is called a homothety.
In estimates, $C$ denotes a positive constant whose dependence is
specified locally and whose value may change between occurrences.
The notation $f\sim h$ means $f/h\to1$ in the stated limit.

\Needspace*{6\baselineskip}
\subsection{Classical formulas used below}

The contracted Bianchi identity gives
\begin{align}
 (\nabla_g)^i(\Ric_g)_{ij}&=\frac12(\nabla_g)_jR_g,\label{id:bianchi-ric}\\
 (\nabla_g)^i(A_g)_{ij}&=(\nabla_g)_j\sigma_1(A_g)
       =\frac1{2(n-1)}(\nabla_g)_jR_g,\label{id:bianchi-schouten}\\
 (\nabla_g)^i(E_g)_{ij}&=\frac{n-2}{2n}(\nabla_g)_jR_g.\label{id:bianchi-tracefree}
\end{align}
These formulas hold for every smooth metric in dimension $n\ge3$;
see also Li--Wei~\cite{LiWei25}.

For a symmetric two-tensor $T$ and a smooth function $f$,
\begin{equation}\label{id:tensor-product}
 \operatorname{div}_g(T\nabla_g f)
 =\langle\operatorname{div}_gT,\nabla_g f\rangle_g
   +\langle T,\nabla_g^2f\rangle_g.
\end{equation}
When $M$ is closed, the divergence theorem and integration by parts give
\begin{equation}\label{id:green}
 \int_M\operatorname{div}_gX\,\dd v_g=0,\qquad
 \int_M f\Delta_g h\,\dd v_g
 =-\int_M\langle\nabla_g f,\nabla_g h\rangle_g\,\dd v_g.
\end{equation}

For $\widehat g=e^{2f}g$, we use the conformal transformation
laws~\cite[Chapter~1]{Besse87}:
\begin{align}
 A_{\widehat g}&=A_g-\nabla_g^2 f+\dd f\otimes\dd f
                   -\frac12|\nabla_g f|_g^2g,
                   \label{id:conformal-schouten}\\
 \sigma_1(A_{\widehat g})&=e^{-2f}\left(
       \sigma_1(A_g)-\Delta_g f-\frac{n-2}{2}|\nabla_g f|_g^2\right),
                   \label{id:conformal-trace}\\
 \Delta_{\widehat g}h&=e^{-2f}\left(
       \Delta_g h+(n-2)\langle\nabla_g f,\nabla_g h\rangle_g\right),
                   \label{id:conformal-laplacian}\\
 \dd v_{\widehat g}&=e^{nf}\dd v_g.
                   \label{id:conformal-volume}
\end{align}
All derivatives on the right are taken with respect to $g$.
The definitions in~\eqref{intro:schouten-defs}, \eqref{intro:Q},
and \eqref{id:notation} give the Schouten decomposition and
the scalar form of the mixed curvature; see also~\cite{LiWei25}:
\begin{align}
 \sigma_2(A_g)&=\frac{R_g^2}{8n(n-1)}-\frac{|E_g|^2}{2(n-2)^2},
       \label{id:sigma}\\
 2(n-1)I_a(g)&=-\Delta_g R_g-\frac{(n-1)(a+4)}{(n-2)^2}|E_g|^2
       +\frac{(n-1)a+n^2-4}{4n(n-1)}R_g^2.
       \label{id:mixed-curvature}
\end{align}
For $g=u^2g_0$ with $g_0$ Einstein, one also has the conformal
identities~\cite{LiWei25}; see also
Ge--Wang--Wei~\cite[Equations~(3.15)--(3.16)]{GWWQuot26}:
\begin{align}
 (\nabla_g^2 u)^\circ&=-\frac{1}{n-2}uE_g,\label{id:hess}\\
 \Delta_g u&=\frac{n}{2}u^{-1}|\nabla_g u|^2-\frac{1}{2(n-1)}uR_g
                  +\frac{1}{2(n-1)}u^{-1}R_{g_0}.\label{id:trace}
\end{align}
All derivatives in these formulas use the target metric $g$.
Indeed, applying \eqref{id:conformal-schouten} to $g_0=u^{-2}g$ gives
\begin{equation}\label{id:einstein-conformal}
 \frac{R_{g_0}}{2n(n-1)}u^{-2}g
 =A_g+u^{-1}\nabla_g^2u-\frac12u^{-2}|\nabla_g u|^2g.
\end{equation}
Taking the trace-free part and the trace of
\eqref{id:einstein-conformal} with respect to $g$, and using
\eqref{id:notation}, gives \eqref{id:hess} and \eqref{id:trace}, respectively.

The product rule~\eqref{id:tensor-product}, together with
\eqref{id:bianchi-tracefree} and \eqref{id:hess}, gives
\begin{equation}\label{id:obata-divergence}
 \operatorname{div}_g(E_g\nabla_g u)
 =\frac{n-2}{2n}\langle\nabla_g R_g,\nabla_g u\rangle_g
                          -\frac{1}{n-2}u|E_g|^2.
\end{equation}
On a closed manifold, integrating~\eqref{id:obata-divergence}
and using~\eqref{id:green} yields the classical Obata identity
\cite{LiWei25}:
\begin{equation}\label{id:obata}
 \int_Mu|E_g|^2\,\dd v_g
 =\frac{(n-2)^2}{2n}
        \int_M\langle\nabla_g R_g,\nabla_g u\rangle_g\,\dd v_g.
\end{equation}
No scalar-curvature sign assumption is needed.

\subsection{Introducing a constant reference curvature}
The first Newton tensor and its divergence are
\begin{equation}\label{id:newton}
 T_1(A_g)\coloneqq \sigma_1(A_g)g-A_g=\frac{R_g}{2n}g-\frac{E_g}{n-2},
 \qquad \operatorname{div}_gT_1(A_g)=0.
\end{equation}
The decomposition in~\eqref{id:newton} follows from
\eqref{id:notation}, and its divergence identity follows from
\eqref{id:bianchi-schouten} for every smooth metric.
Related conformal divergence formulas involving Newton tensors
are given by Ge--Wang--Wei~\cite[Proposition~5.3]{GWWFlow26};
their $k=2$ formula likewise requires no local conformal flatness.
Introducing a real constant $r$ yields the factor $R_g-r$.
When $M$ is closed, the choice $r\coloneqq\min_MR_g$ makes
this factor nonnegative.

\begin{proposition}[Newton identity with a reference curvature]\label{id:shifted}
For every real constant $r$,
\begin{equation}\label{id:shifted-formula}
\begin{split}
 \operatorname{div}_g\!\left[(n-1)\left(T_1(A_g)-\frac{r}{2n}g\right)\nabla_g u\right]={}&
 \frac{R_g-r}{4n}\left(n(n-1)u^{-1}|\nabla_g u|^2+ur+u^{-1}R_{g_0}\right)\\
 &+u\left(\frac{r^2}{4n}-2(n-1)\sigma_2(A_g)\right).
\end{split}
\end{equation}
\end{proposition}
\begin{proof}
Since $r$ is constant, \eqref{id:newton} shows that
$T_1(A_g)-rg/(2n)$ is divergence-free. Applying
\eqref{id:tensor-product} and \eqref{id:hess}, with the
decomposition in~\eqref{id:newton}, gives
\begin{equation}\label{id:shifted-divergence}
\begin{split}
 &\operatorname{div}_g\!\left[(n-1)
       \left(T_1(A_g)-\frac r{2n}g\right)\nabla_g u\right]\\
 &\qquad=\frac{n-1}{2n}(R_g-r)\Delta_g u
       +\frac{n-1}{(n-2)^2}u|E_g|^2.
\end{split}
\end{equation}
Substituting~\eqref{id:trace} and \eqref{id:sigma} into
\eqref{id:shifted-divergence} gives~\eqref{id:shifted-formula}.
\end{proof}

\Needspace*{6\baselineskip}
\begin{proposition}[Pointwise mixed-curvature identity]\label{id:mixed}
Let $(M^n,g_0)$ be a smooth Einstein manifold of dimension $n\ge4$,
let $u\in C^\infty(M)$ be positive, and set $g\coloneqq u^2g_0$.
For every pair of real constants $a$ and $r$, the following identity
holds pointwise on $M$, with all differential operators and
contractions taken with respect to $g$:
\begin{equation}\label{id:mixed-formula}
\begin{split}
 &\operatorname{div}_g\!\left[u\nabla_g R_g
 -\frac{(n-1)a+n^2+2n-4}{n-2}E_g\nabla_g u
 +\frac{(n-1)a+n^2-4}{2n}(R_g-r)\nabla_g u\right]\\
 &\quad=\frac n{n-2}u|E_g|^2
 +\frac{(n-1)a+n^2-4}{4n(n-1)}(R_g-r)
       \left(n(n-1)u^{-1}|\nabla_g u|^2+ur+u^{-1}R_{g_0}\right)\\
 &\qquad\quad+u\left[\frac{(n-1)a+n^2-4}{4n(n-1)}r^2-2(n-1)I_a(g)\right].
\end{split}
\end{equation}
No additional differential equation is imposed on $u$ or $g$;
in particular, $I_a(g)$ need not be constant. Neither compactness
nor a sign assumption on $R_g$ or $R_{g_0}$ is required.
If $a$ remains constant and $r\in C^1(M)$ is allowed to vary, add
$-\frac{(n-1)a+n^2-4}{2n}\langle\nabla_g r,\nabla_g u\rangle_g$
to the right side.
\end{proposition}
\begin{proof}
By~\eqref{id:newton}, the vector field on the left side
of~\eqref{id:mixed-formula} equals
\begin{equation}\label{id:current-decomposition}
 u\nabla_g R_g-\frac{2n}{n-2}E_g\nabla_g u
 +\bigl((n-1)a+n^2-4\bigr)
        \left(T_1(A_g)-\frac r{2n}g\right)\nabla_g u.
\end{equation}
By the product rule~\eqref{id:tensor-product} and
\eqref{id:obata-divergence}, the pointwise Obata identity takes
the divergence form
\begin{equation}\label{id:obata-current}
 \operatorname{div}_g\!\left(u\nabla_g R_g-\frac{2n}{n-2}E_g\nabla_g u\right)
 =u\Delta_g R_g+\frac{2n}{(n-2)^2}u|E_g|^2.
\end{equation}
Take the divergence of~\eqref{id:current-decomposition} and use
\eqref{id:obata-current} and \eqref{id:shifted-formula}.
Substitute~\eqref{id:mixed-curvature} for $\Delta_g R_g$ and
\eqref{id:sigma} for $\sigma_2(A_g)$.
The $uR_g^2$ terms cancel, and the coefficient of $u|E_g|^2$ becomes
\[
 \frac{2n-(n-1)(a+4)+(n-1)a+n^2-4}{(n-2)^2}
 =\frac n{n-2}.
\]
This gives~\eqref{id:mixed-formula}. For variable $r$, applying
the product rule~\eqref{id:tensor-product} to the last term
of~\eqref{id:current-decomposition} gives the stated additional term.
\end{proof}

\subsection{Relation to Gursky-type identities}\label{id:comparison-section}
The higher-order Case--Gursky--V\'etois identity appears in
\cite[Theorem~2.1]{LiWei25} and, for constant $I_a(g)$, in
\cite[Lemma~3.2]{Case24}. It contains the mixed contraction
$E_g(\nabla_g R_g,\nabla_g u)$, whose estimate leads to the parameter
intervals in those papers. The decomposition
\eqref{id:current-decomposition} combines the lower-order Obata
divergence with the Newton identity to avoid the mixed contraction
$E_g(\nabla_g R_g,\nabla_g u)$.

In the next section, under the
hypotheses of Theorem~\ref{main:rigidity}, we choose
$r\coloneqq\min_MR_g$. Evaluating~\eqref{id:mixed-curvature} at a point
attaining this minimum gives~\eqref{rig:min}. This inequality makes
every term obtained by integrating~\eqref{id:mixed-formula}
nonnegative for $a\ge-4$.
\section{Rigidity from the scalar-curvature minimum}\label{sec:rigidity}

For the scalar equation~\eqref{id:mixed-curvature}, we use
the coefficients of Li--Wei~\cite{LiWei25}:
\begin{equation}\label{rig:coeff}
 \alpha_1\coloneqq \frac{(n-1)(a+4)}{(n-2)^2},\qquad
 \alpha_2\coloneqq \frac{(n-1)a+n^2-4}{4n(n-1)}.
\end{equation}
These are the coefficients of $-|E_g|^2$ and $R_g^2$ in
\eqref{id:mixed-curvature}, respectively.
For $a=0$, they agree with the coefficients in the scalar form
of $Q$ used by Ge--Wang--Wei~\cite[Equations~(3.13)--(3.14)]{GWWQuot26}.
For $n\ge4$ and $a\ge-4$, both coefficients are nonnegative.
Within this range, $\alpha_2=0$ precisely when $n=4$ and $a=-4$.

\begin{proof}[Proof of Theorem~\ref{main:rigidity}]
We take the scalar-curvature minimum as the reference curvature.
Set $r\coloneqq\min_MR_g\ge0$, and choose $x_0\in M$ with
$R_g(x_0)=r$. Then $\nabla_gR_g(x_0)=0$ and
$\Delta_gR_g(x_0)\ge0$.
Evaluating~\eqref{id:mixed-curvature} at this point and using
the signs of the coefficients in~\eqref{rig:coeff} gives
\begin{equation}\label{rig:min}
 2(n-1)\Lambda=-\Delta_g R_g(x_0)-\alpha_1|E_g(x_0)|^2+\alpha_2r^2\le \alpha_2r^2.
\end{equation}
Integrating~\eqref{id:mixed-formula} and using~\eqref{id:green} gives
\begin{equation}\label{rig:integral}
\begin{split}
0=\int_M\Bigg\{&\frac n{n-2}u|E_g|^2
 +\alpha_2(R_g-r)\left(n(n-1)u^{-1}|\nabla_g u|^2+ur+u^{-1}R_{g_0}\right)\\
 &+u\bigl[\alpha_2r^2-2(n-1)\Lambda\bigr]\Bigg\}\,\dd v_g.
\end{split}
\end{equation}
By~\eqref{rig:min}, \eqref{rig:coeff}, and the hypotheses, every
summand in~\eqref{rig:integral} is nonnegative, including when
$r=0$ or $\alpha_2=0$. Since $u>0$, \eqref{rig:integral} gives
$E_g=0$. Equation~\eqref{id:bianchi-tracefree} then gives
$\dd R_g=0$, so $R_g$ is constant because $M$ is connected.

At a maximum point of $u$, we have $\nabla_g u=0$ and
$\Delta_g u\le0$. Thus~\eqref{id:trace} yields
$R_g\ge R_{g_0}(\max_M u)^{-2}>0$.
Obata's classification~\cite{Obata62,Obata71} gives the stated conclusion.
With $E_g=0$ and $R_g$ constant,~\eqref{id:mixed-curvature} gives
\eqref{main:rigidity-value}.
\end{proof}

The case $a=0$ and $0<\theta\le1$ of the following corollary
was proved by Ge--Wang--Wei~\cite[Theorem~3.2]{GWWQuot26}
in dimensions $n\ge3$. For $n\ge4$, our argument gives a new
proof of their rigidity theorem using~\eqref{id:mixed-formula},
and extends it to $a\ge-4$ and $\theta\le1$.

\begin{corollary}[Mixed-curvature quotients]\label{rig:quotient}
Let $(M^n,g_0)$ be closed and connected, $n\ge4$, with $g_0$
Einstein and $R_{g_0}>0$. Suppose $g=u^2g_0$, with $u>0$ smooth, satisfies
\begin{equation}\label{rig:quotient-equation}
 R_g>0,\qquad I_a(g)=\Lambda R_g^\theta,\qquad a\ge-4,\quad \theta\le1.
\end{equation}
Here $\Lambda$ and $\theta$ are real constants.
Then $g$ has the form described in Theorem~\ref{main:rigidity}, and
\begin{equation}\label{clo:mixed-scale}
 \Lambda=\frac{(n-1)a+n^2-4}{8n(n-1)^2}R_g^{\,2-\theta}.
\end{equation}
In particular, $\Lambda>0$ unless $n=4$ and $a=-4$, in which case
$\Lambda=0$.
\end{corollary}
\begin{proof}
If $\alpha_2=0$, then \eqref{rig:coeff} gives $n=4$, $a=-4$,
and $\alpha_1=0$. Integrating~\eqref{id:mixed-curvature} with
\eqref{rig:quotient-equation} and \eqref{id:green} gives
$\Lambda=0$, since $R_g^\theta>0$. Equation~\eqref{id:mixed-curvature}
then makes $R_g$ harmonic and hence constant, and
\eqref{id:obata} gives $E_g=0$. Obata's theorem~\cite{Obata62,Obata71}
gives the classification, proving the assertion in this case.

Suppose $\alpha_2>0$. Set $r\coloneqq\min_MR_g>0$ and choose
$x_0$ with $R_g(x_0)=r$, so $\Delta_gR_g(x_0)\ge0$.
Evaluating~\eqref{id:mixed-curvature} at $x_0$, using
\eqref{rig:quotient-equation} and $\alpha_1\ge0$ from
\eqref{rig:coeff}, and dividing by $r$ gives
\begin{equation}\label{rig:quotient-min}
 2(n-1)\Lambda r^{\theta-1}\le\alpha_2r.
\end{equation}
Substituting~\eqref{rig:quotient-equation} into
\eqref{id:mixed-formula}, integrating using~\eqref{id:green},
and combining the terms involving $r$ gives
\begin{equation}\label{rig:quotient-integral}
\begin{split}
0=\int_M\Bigg\{&\frac n{n-2}u|E_g|^2
 +\alpha_2(R_g-r)\left(n(n-1)u^{-1}|\nabla_g u|^2+u^{-1}R_{g_0}\right)\\
 &+uR_g\bigl[\alpha_2r-2(n-1)\Lambda R_g^{\theta-1}\bigr]\Bigg\}\,\dd v_g.
\end{split}
\end{equation}
The first two terms in~\eqref{rig:quotient-integral} are nonnegative.
If $\Lambda\le0$, its last bracket is strictly positive, a
contradiction. Thus $\Lambda>0$.
Since $\theta\le1$ and $R_g\ge r>0$, we have
$R_g^{\theta-1}\le r^{\theta-1}$. Together with
\eqref{rig:quotient-min}, this makes every summand in
\eqref{rig:quotient-integral} nonnegative and gives $E_g=0$.
Equation~\eqref{id:bianchi-tracefree} then implies that $R_g$ is constant.
Obata's theorem~\cite{Obata62,Obata71} gives the classification.
Substitution into~\eqref{id:mixed-curvature}, using
\eqref{rig:quotient-equation}, yields~\eqref{clo:mixed-scale}.
\end{proof}

\begin{remark}
Under the same assumptions on $M$, $g_0$, and $u$ as in
Corollary~\ref{rig:quotient}, suppose $R_g>0$, $a\ge-4$,
and $I_a(g)=F(R_g)$. The same metric classification holds if
$F$ is continuous and $F(t)/t$ is nonincreasing on $R_g(M)$.
Indeed, evaluating~\eqref{id:mixed-curvature} at a point where
$R_g$ attains its minimum $r>0$, and using~\eqref{rig:coeff}, gives
\begin{equation}\label{rig:monotone-min}
 2(n-1)\frac{F(r)}r\le\alpha_2r.
\end{equation}
Integrating~\eqref{id:mixed-formula} using~\eqref{id:green} gives
\eqref{rig:quotient-integral} with $\Lambda R_g^{\theta-1}$
replaced by $F(R_g)/R_g$. Since $F(R_g)/R_g\le F(r)/r$,
\eqref{rig:monotone-min} makes every term nonnegative, so $E_g=0$.
Equation~\eqref{id:bianchi-tracefree} shows that $R_g$ is constant, and
Obata's theorem~\cite{Obata62,Obata71} gives the classification.
\end{remark}

\begin{proposition}[Rigidity without a scalar-curvature sign assumption]\label{four:unsigned}
Let $(M^n,g_0)$ be closed and connected, $n\ge4$, with $g_0$
Einstein and $R_{g_0}>0$. If a smooth metric $g=u^2g_0$, $u>0$,
satisfies $I_a(g)=\Lambda$ for a real constant $\Lambda$ and
\begin{equation}\label{rig:unsigned-range}
 -4\le a\le-\frac{2(n-2)}{n-1},
\end{equation}
then $R_g>0$ and $g$ has the form described in
Theorem~\ref{main:rigidity}.
\end{proposition}
\begin{proof}
We first determine the sign of $\Lambda$ using the weighted
identity in Li--Wei~\cite[(3.1)]{LiWei25}, which we derive below.
We then use Case's conformal-Laplacian argument~\cite[Lemma~4.3]{Case24}
to obtain $R_g>0$.
With $\alpha_1,\alpha_2$ as in~\eqref{rig:coeff}, set
\begin{equation}\label{rig:unsigned-coeff}
 \kappa\coloneqq\frac{-(n-1)a-2(n-2)}{(n-2)^2},\qquad
 c_n\coloneqq\frac{n-2}{4(n-1)},\qquad
 c_n-\alpha_2=\frac{(n-2)^2\kappa}{4n(n-1)}.
\end{equation}
Equations~\eqref{rig:unsigned-range}, \eqref{rig:coeff}, and
\eqref{rig:unsigned-coeff} give $\alpha_1,\alpha_2,\kappa\ge0$.
Multiplying~\eqref{id:mixed-curvature} by $u$, integrating using
\eqref{id:green}, and applying~\eqref{id:obata} gives
\begin{equation}\label{rig:unsigned-weighted}
\begin{split}
 2(n-1)\Lambda\int_Mu\,\dd v_g
 &=\left(\frac{2n}{(n-2)^2}-\alpha_1\right)
                   \int_Mu|E_g|^2\,\dd v_g
             +\alpha_2\int_MuR_g^2\,\dd v_g\\
 &=\kappa\int_Mu|E_g|^2\,\dd v_g
             +\alpha_2\int_MuR_g^2\,\dd v_g,
\end{split}
\end{equation}
where the second equality uses~\eqref{rig:coeff} and
\eqref{rig:unsigned-coeff}. At a maximum point $x_{\max}$ of $u$,
\eqref{id:trace} gives
\begin{equation}\label{rig:positive-square}
 R_g(x_{\max})\ge R_{g_0}u(x_{\max})^{-2}>0,
 \qquad \int_MuR_g^2\,\dd v_g>0.
\end{equation}

If $(n,a)=(4,-4)$, then \eqref{rig:coeff} gives
$\alpha_1=\alpha_2=0$. Integrating~\eqref{id:mixed-curvature}
using~\eqref{id:green} gives $\Lambda=0$.
Since $\kappa=2$ by~\eqref{rig:unsigned-coeff},
\eqref{rig:unsigned-weighted} gives $E_g=0$.
Equation~\eqref{id:bianchi-tracefree} shows that $R_g$ is constant,
and \eqref{rig:positive-square} gives $R_g>0$.

For all other parameters in~\eqref{rig:unsigned-range},
$\alpha_2>0$ by~\eqref{rig:coeff}. Thus
\eqref{rig:unsigned-weighted} and \eqref{rig:positive-square}
give $\Lambda>0$. For a metric $h$, write
$L_h\coloneqq-\Delta_h+c_nR_h$.
Equations~\eqref{id:mixed-curvature}, \eqref{rig:coeff}, and
\eqref{rig:unsigned-coeff} yield
\begin{equation}\label{rig:unsigned-laplacian}
 L_gR_g=2(n-1)\Lambda+\alpha_1|E_g|^2
       +\frac{(n-2)^2\kappa}{4n(n-1)}R_g^2>0.
\end{equation}
The conformal covariance of the conformal Laplacian
\cite[(2.7)]{LeeParker87}, together with
\eqref{rig:unsigned-laplacian}, gives
\begin{equation}\label{rig:unsigned-covariance}
 L_{g_0}\bigl(u^{(n-2)/2}R_g\bigr)
 =u^{(n+2)/2}L_gR_g>0.
\end{equation}
Since $c_nR_{g_0}>0$, a nonpositive minimum of
$u^{(n-2)/2}R_g$ would make the left side of
\eqref{rig:unsigned-covariance} nonpositive at that point.
This contradiction proves $R_g>0$.
Theorem~\ref{main:rigidity} applies in both cases.
\end{proof}
\section{The radial equation and smooth spherical poles}\label{sec:radial}

The construction has two geometric models. Near the equator, the
metric is a perturbation of a suitably rescaled Schwarzschild neck;
near the two poles, it is a perturbation of scaled round caps.
We perturb these model families and tune the shooting data so that
the neck trajectory and a smooth cap trajectory meet at a common
phase-space section. Uniqueness for the same curvature ODE then
identifies the two trajectories; after matching their scale at that
section, this gives a smooth metric satisfying
the constant-$I_a$ equation across the matching region. Reflection
across the equator supplies the second cap. The Schwarzschild model
and the matching argument are developed in
Section~\ref{sec:high-neck}, in particular Lemma~\ref{shooting};
the final scale is fixed by~\eqref{sharp:normalization}.

Throughout this section and Section~\ref{sec:high-neck}, fix
$n\ge4$. We first derive the radial ODE for $I_a(g)=\Lambda$.
We then construct the smooth cap family by replacing the infinite
cylindrical end at the south pole with the squared radius of the
spherical slices. The reconstruction proof checks smoothness of
the conformal factor in the original sphere coordinates, and the
last two lemmas describe the cap data and scalar-curvature positivity.

Set
\begin{equation}\label{radial:B-definition}
 B\coloneqq(n-1)a+n^2-4.
\end{equation}
Identity~\eqref{id:mixed-curvature}, together with
$E_g=(n-2)A_g^\circ$ from~\eqref{id:notation} and
\eqref{radial:B-definition}, gives
\begin{equation}\label{ode:scalar-form}
 I_a(g)=-\Delta_g \sigma_1(A_g)-\frac{a+4}{2}|A_g^\circ|_g^2+\frac{B}{2n}\sigma_1(A_g)^2.
\end{equation}

We first record the standard round curvature values used below;
see also \cite[Section~2.5]{GWWQuot26}.
\begin{lemma}[Curvature values for the unit round sphere]\label{ode:round-curvatures}
The unit round metric $g_{\Sph^n}$ has sectional curvature $1$
and the following curvature values:
\begin{equation}\label{radial:round-curvature-values}
 \begin{gathered}
 R_{g_{\Sph^n}}=n(n-1),\qquad
 A_{g_{\Sph^n}}=\frac12g_{\Sph^n},\\
 \sigma_1(A_{g_{\Sph^n}})=\frac n2,\qquad
 \sigma_2(A_{g_{\Sph^n}})=\frac{n(n-1)}8,\qquad
 Q_{g_{\Sph^n}}=\frac{n(n^2-4)}8.
 \end{gathered}
\end{equation}
In particular, equations~\eqref{intro:Q},
\eqref{radial:B-definition}, and~\eqref{radial:round-curvature-values} give
\begin{equation}\label{radial:round-mixed-curvature}
 I_a(g_{\Sph^n})=\frac n8\bigl(n^2-4+(n-1)a\bigr)=\frac{nB}{8}.
\end{equation}
These values hold at every point of the sphere, including the poles.
\end{lemma}

\subsection{Coordinates and the reduced equation}

Write the unit sphere as
\[
 \Sph^n
 =
 \{(\sin\theta\,\Theta,\cos\theta):
 0\leq\theta\leq\pi,\ \Theta\in\Sph^{n-1}\}.
\]
Thus
\begin{align}
	P_N&=(0,\ldots,0,1)
 \quad\Longleftrightarrow\quad
 \theta=0,\\
 P_S&=(0,\ldots,0,-1)
 \quad\Longleftrightarrow\quad
 \theta=\pi.
 \end{align}
Away from the two poles, the round metric is
\[
 g_{\Sph^n}
 =
 \dd\theta^2+\sin^2\theta\,g_{\Sph^{n-1}}.
\]

Introduce
\[
 t:=\log\tan(\theta/2).
\]
Then $t$ ranges over $\R$, with
\[
 t\to-\infty\quad\text{at }P_N,
 \qquad
 t\to+\infty\quad\text{at }P_S,
\]
and the equator $\theta=\pi/2$ corresponds to $t=0$.
Since
\[
 \dd\theta=\sin\theta\,\dd t,
 \qquad
 \sin\theta = \frac{1}{\cosh t}=\sech t,
\]
we obtain
\[
 g_{\Sph^n}
 =
 \sech^2t
 \bigl(\dd t^2+g_{\Sph^{n-1}}\bigr).
\]
Thus the sphere with its two poles removed is conformal to the
product cylinder $\R\times\Sph^{n-1}$.

Now let $g=u(\theta)^2g_{\Sph^n}$, where $u>0$ depends only on
$\theta$. In the same coordinates,
\[
 g=\bigl(u(\theta(t))\sech t\bigr)^2
   \bigl(\dd t^2+g_{\Sph^{n-1}}\bigr).
\]
Define $v(t)$ by $e^{v(t)}=u(\theta(t))\sech t$. Taking the
logarithm of this positive factor yields
\begin{equation}\label{cylinder}
 g=e^{2v(t)}(\dd t^2+g_{\Sph^{n-1}}),\qquad
 v(t):=\log u(\theta(t))-\log\cosh t.
\end{equation}
The limits $t\to-\infty$ and $t\to+\infty$ correspond to the
poles $\theta=0$ and $\theta=\pi$, respectively. For the round
metric, $u\equiv1$, so
\begin{align}\label{09092223}
v_{\rm rd}(t)=-\log\cosh t.
\end{align}
A prime denotes differentiation in $t$. Set
\begin{equation}\label{variables}
 x\coloneqq-v',\quad y\coloneqq x',\quad z\coloneqq y',\quad q\coloneqq1-x^2,\quad
 J\coloneqq y+\frac{n-2}{2}q.
\end{equation}
For each $t$, the spherical slice $\{t\}\times\Sph^{n-1}$ has
induced metric $e^{2v(t)}g_{\Sph^{n-1}}$ and radius $e^{v(t)}$.
In this paper, a \emph{neck} is a neighborhood of a slice where
this radius has a positive strict local minimum.
That slice is the \emph{neck center}.
The necks constructed below have a nondegenerate minimum at $t=0$:
\[
 v'(0)=0,\qquad v''(0)>0,
 \qquad\text{equivalently}\qquad x(0)=0,\quad y(0)<0.
\]
Their neck radius is $e^{v(0)}$.

The next lemma expresses the curvatures of $g$ in the variables
\eqref{variables}, reducing $I_a(g)=\Lambda$ to an ODE.
\begin{lemma}[Radial curvature formulas]\label{ode:curvature-formulas}
Let $g$ be a smooth metric of the form~\eqref{cylinder} for $t$
in an open interval, with $x,y,z,q,J$ as in~\eqref{variables}.
Let $K_{\rm rad}$ denote the sectional curvature of a two-plane
containing the radial direction $\partial_t$, and let $K_{\rm tan}$
denote that of a two-plane tangent to a spherical slice.
Rotational symmetry makes both functions of $t$ alone.
The Schouten eigenvalues in the radial and tangential directions are
\begin{equation}\label{radial:schouten-eigenvalues}
 \lambda_{\rm rad}=e^{-2v}(y-q/2),\qquad
 \lambda_{\rm tan}=e^{-2v}q/2,
\end{equation}
with multiplicities $1$ and $n-1$, respectively. Moreover,
\begin{equation}\label{radial-curvatures}
\begin{gathered}
 \sigma_1(A_g)=e^{-2v}J,\qquad K_{\rm rad}=e^{-2v}y,\qquad K_{\rm tan}=e^{-2v}q,\\
 e^{4v}|A_g^\circ|_g^2=\frac{n-1}{n}(y-q)^2,\qquad
 e^{4v}\sigma _2(A_g)=(n-1)\left(\frac{qy}{2}+\frac{n-4}{8}q^2\right).
\end{gathered}
\end{equation}
For every smooth radial function $f$,
\begin{equation}\label{radial:laplacian}
 \Delta_g f=e^{-2v}\bigl(f''-(n-2)xf'\bigr).
\end{equation}
Finally,
\begin{equation}\label{radial:mixed-curvature}
 \begin{split}
 e^{4v}I_a(g)={}&-z'+2(n-4)xz+\frac{3(n-4)}2y^2 +(-n^2+10n-20)x^2y+\left(\frac B2-2n+6\right)qy\\
 &+(n-4)(n-2)x^2q+\frac{(n-4)(B+4)}8q^2.
 \end{split}
\end{equation}
\end{lemma}
\begin{proof}
For the product cylinder $h=\dd t^2+g_{\Sph^{n-1}}$, the line
factor is flat and the spherical factor has Ricci tensor
$(n-2)g_{\Sph^{n-1}}$. The Schouten definition in
\eqref{intro:schouten-defs} gives
\begin{equation}\label{radial:cylinder-schouten}
 \Ric_h=(n-2)g_{\Sph^{n-1}},\quad R_h=(n-1)(n-2),\quad
 A_h=-\tfrac12\dd t^2+\tfrac12g_{\Sph^{n-1}}.
\end{equation}
Applying~\eqref{id:conformal-schouten} to $g=e^{2v}h$ in
\eqref{cylinder}, and using~\eqref{variables} and
\eqref{radial:cylinder-schouten}, gives the eigenvalues
\eqref{radial:schouten-eigenvalues}. Their sum, pairwise products,
and squared deviations from the mean give the three Schouten
formulas in~\eqref{radial-curvatures}.

Since $g$ is locally conformally flat by~\eqref{cylinder}, the
Weyl--Schouten decomposition gives
\begin{equation}\label{radial:sectional-decomposition}
 K_g\bigl(\operatorname{span}\{U,V\}\bigr)
 =A_g(U,U)+A_g(V,V)
\end{equation}
for $g$-orthonormal $U,V$. Substitution of
\eqref{radial:schouten-eigenvalues} into
\eqref{radial:sectional-decomposition} gives $K_{\rm rad}$ and
$K_{\rm tan}$ in~\eqref{radial-curvatures}.

The volume density and inverse radial coefficient of
\eqref{cylinder} yield
\begin{equation}\label{radial:laplacian-divergence}
 \Delta_g f=e^{-nv}\bigl(e^{(n-2)v}f'\bigr)'.
\end{equation}
Using $v'=-x$ from~\eqref{variables} in
\eqref{radial:laplacian-divergence} proves~\eqref{radial:laplacian}.
Applying~\eqref{radial:laplacian} to
$\sigma_1(A_g)=e^{-2v}J$ from~\eqref{radial-curvatures}, with
$v'=-x$ and $x'=y$ from~\eqref{variables}, gives
\begin{equation}\label{radial:trace-laplacian}
 \Delta_g\sigma_1(A_g)
 =e^{-4v}\bigl[J''+(6-n)xJ'+2yJ-2(n-4)x^2J\bigr].
\end{equation}
Substituting~\eqref{radial:trace-laplacian} and
\eqref{radial-curvatures} into~\eqref{ode:scalar-form} yields
\begin{equation}\label{radial:trace-equation}
 \begin{split}
 e^{4v}I_a(g)={}&-J''-(6-n)xJ'-\bigl(2y-2(n-4)x^2\bigr)J\\
 &-\frac{(n-1)(a+4)}{2n}(y-q)^2+\frac B{2n}J^2.
 \end{split}
\end{equation}
Differentiating $J$ in~\eqref{variables} gives
\begin{equation}\label{radial:trace-derivatives}
 J'=z-(n-2)xy,\qquad J''=z'-(n-2)(y^2+xz).
\end{equation}
Substituting~\eqref{variables} and~\eqref{radial:trace-derivatives}
into~\eqref{radial:trace-equation}, and using
$(n-1)(a+4)=B-n(n-4)$ from~\eqref{radial:B-definition}, proves~\eqref{radial:mixed-curvature}.
\end{proof}

Following Gursky--Malchiodi's four-dimensional reduction
\cite[Sections~4.1 and~5.2]{GM12}, define
\begin{equation}\label{K}
 \cK:=-xz+\tfrac12y^2+(n-4)x^2y-(n-2)x^2q-\tfrac18(B+4)q^2.
\end{equation}
For $n=4$ and $a=-4/(1+\beta)$, $6\cK$ is $K_\beta$ in
\cite[(5.12)]{GM12}. A prime on $\cK$ denotes differentiation
along $X(t)=(x(t),y(t),z(t))$. The next lemma establishes
\begin{equation}\label{K:requirements}
 \cK'-(n-4)x\cK=xe^{4v}I_a(g),\qquad \cK(\pm1,0,0)=0.
\end{equation}
The pole condition determines the integration constant in the
first identity. Indeed, for a smooth positive radial conformal
factor $u$, the coordinate relation~\eqref{cylinder} and the
variable definitions~\eqref{variables} give
\begin{equation}\label{09092117}
 \begin{aligned}
 v(t)&=-t+O(1),& (x,y,z)&\longrightarrow(1,0,0)
 &&(t\to+\infty),\\
 v(t)&=t+O(1),& (x,y,z)&\longrightarrow(-1,0,0)
 &&(t\to-\infty).
 \end{aligned}
\end{equation}
Here the first three $t$-derivatives of
$\log u(\theta(t))$ tend to zero, since $u$ is smooth at the
poles and $\theta'=\sech t$ and its derivatives tend to zero.

\begin{lemma}[Radial reduction; cf.~{\cite[Section~5.2]{GM12}} in dimension four]\label{ode:reduction}
Let $g=u^2g_{\Sph^n}$ be smooth on $\Sph^n$, with $u$ positive
and radial. If $I_a(g)=\Lambda$ is constant, then the variables
in~\eqref{variables} satisfy $\Lambda e^{4v}=-n\cK$ and the
autonomous system
\begin{equation}\label{system}
 \begin{split}
 x'&=y,\qquad y'=z,\qquad z'=G(x,y,z),\\
 G&:=(n-8)xz+2(n-3)y^2+(6n-20)x^2y\vphantom{\tfrac12}\\
 &\quad +(B/2-2n+6)qy-4(n-2)x^2q-\tfrac12(B+4)q^2.
 \end{split}
\end{equation}
Along every solution of this system, with $v'=-x$,
\begin{equation}\label{transport}
 \cK'=-4x\cK,\qquad e^{-4v}\cK=\text{constant}.
\end{equation}
Conversely, fix a solution $X(t)=(x(t),y(t),z(t))$ of
\eqref{system} on an interval containing $0$. For any $v_0\in\R$, define
\begin{equation}\label{radial:v-reconstruction}
 v(t):=v_0-\int_0^t x(w)\,\dd w.
\end{equation}
Then $v(0)=v_0$ and $v'=-x$, and the metric~\eqref{cylinder}
satisfies $I_a(g)=-ne^{-4v_0}\cK(0)$ throughout that interval.
The polynomial also has the equivalent expression
\begin{equation}\label{K-J}
 \cK=-xJ'+\tfrac12J^2-(2x^2+\frac{n-2}{2}q)J-\frac{(n-1)(a+4)}8q^2.
\end{equation}
\end{lemma}
\begin{proof}
For $n=4$ and $a\ne0$, the vector field in~\eqref{system}
and the transport law~\eqref{transport} agree with
\cite[(5.10), (5.12), and (5.14)]{GM12} under the substitutions
$\beta=-1-4/a$ and $K_\beta=6\cK$. We now verify the formulas
for every $n\ge4$.

For any smooth function $v=v(t)$, formula~\eqref{radial:mixed-curvature}
and the definitions of $\cK$ and $G$ in~\eqref{K} and~\eqref{system}
give
\begin{equation}\label{radial:equivalence}
 e^{4v}I_a(g)+n\cK=G(x,y,z)-z'.
\end{equation}
Direct differentiation of~\eqref{K}, using $x'=y$ and $y'=z$ from~\eqref{variables},
gives, for every smooth function $v=v(t)$,
\begin{align}\label{09092216}
 \cK'+4x\cK=x\bigl(G(x,y,z)-z'\bigr).
\end{align}
Combining \eqref{radial:equivalence} and \eqref{09092216} gives
\[
 \cK'-(n-4)x\cK=xe^{4v}I_a(g).
\]
Also, $\cK(\pm1,0,0)=0$ follows directly from~\eqref{K}.
This verifies both requirements in~\eqref{K:requirements}.
Multiplying~\eqref{K:requirements} by $e^{(n-4)v}$ and using
$v'=-x$ from~\eqref{variables} gives
\begin{equation}\label{09092213}
 \bigl(e^{(n-4)v}\cK\bigr)'=xe^{nv}I_a(g).
\end{equation}
If $I_a(g)=\Lambda$, then~\eqref{09092213} and
$v'=-x$ from~\eqref{variables} yield
\begin{equation}\label{first-integral}
 \left(e^{(n-4)v}\cK+\frac\Lambda n e^{nv}\right)'=0.
\end{equation}
By~\eqref{09092117} and $\cK(1,0,0)=0$ in
\eqref{K:requirements}, both terms inside the derivative in
\eqref{first-integral} tend to zero at the south pole; here
$n\ge4$. Thus
\begin{equation}\label{09281548}
 e^{(n-4)v}\cK+\frac{\Lambda}{n}e^{nv}=0,
 \qquad \Lambda e^{4v}=-n\cK.
\end{equation}
Substitution of~\eqref{09281548} into
\eqref{radial:equivalence} gives $z'=G$.

Along every solution of~\eqref{system}, identity~\eqref{09092216}
gives $\cK'=-4x\cK$. Since $v'=-x$ by
\eqref{radial:v-reconstruction}, this also makes
$e^{-4v}\cK$ constant, proving~\eqref{transport}.
For a solution of~\eqref{system} and $v$ from
\eqref{radial:v-reconstruction}, equation~\eqref{radial:equivalence}
gives $I_a(g)=-ne^{-4v}\cK$; evaluation of~\eqref{transport}
at $t=0$ proves the converse. Finally, substituting
$z=J'+(n-2)xy$ from~\eqref{radial:trace-derivatives} and
$y=J-(n-2)q/2$ from~\eqref{variables} into~\eqref{K}
gives~\eqref{K-J}.
\end{proof}

Solutions of~\eqref{system} depend smoothly on their initial data
and parameters on every compact existence interval;
see~\cite[Sections~2.4 and~2.6]{Teschl}.
Lemma~\ref{ode:reduction} ensures that the initial-value construction
solves the curvature equation. We next construct initial data
for which the reconstructed metric extends smoothly across a pole.

\subsection{A local equation at the south pole}
For the round metric, equation~\eqref{09092223} gives the trajectory
\begin{equation}\label{round}
 X_{\rm rd}(t)=(\tanh t,\sech^2t,-2\tanh t\sech^2t),
\end{equation}
which is the reference trajectory for the initial data curves
in Lemma~\ref{ode:caps} and for the linearization in
Section~\ref{sec:high-neck}.
In particular, the round trajectory has the oriented endpoint data
\[
\begin{array}{c|cc}
 & t\to-\infty & t\to+\infty\\
\hline
\text{spherical pole} & P_N & P_S\\
x & -1 & 1\\
X=(x,y,z) & (-1,0,0) & (1,0,0).
\end{array}
\]
The vector field~\eqref{system} is reversible: if $X$ is a
solution, then so is
\begin{equation}\label{radial:reflection}
 (\mathcal R X)(t):=\bigl(-x(-t),y(-t),-z(-t)\bigr).
\end{equation}
For initial data with $x(0)=z(0)=0$,
uniqueness and $v'=-x$ from~\eqref{variables} make $x,z$ odd
and $y,v$ even. Thus smooth closure
at $P_S$ also gives smooth closure at $P_N$. Throughout the local
pole construction below, we work at the south pole $P_S$,
corresponding to $\theta=\pi$ and $t\to+\infty$; the north pole
$P_N$ is obtained only by reflection.

Define the squared slice radius and curvature variables by
\begin{equation}\label{pole:variables}
 s:=e^{2v},\qquad k:=\frac qs,\qquad m:=\sigma_1(A_g).
\end{equation}
For a smooth south pole, equation~\eqref{09092117} gives
$s\to0$ and $x\to1$ as $t\to+\infty$. Since $s'=-2xs$ by
\eqref{variables} and~\eqref{pole:variables}, $s$ decreases
strictly near the pole and can be used as a local independent
variable. The spherical slices have metric
$s\,g_{\Sph^{n-1}}$ by~\eqref{cylinder}, so $s$ is their
squared radius, not the squared radial coordinate of the
background round sphere.

The change of variable does not by itself prove smoothness at
$P_S$. The first equation below still has the vanishing
coefficient $s$ in front of $k_s$. We first rewrite the degenerate
equation at $s=0$ in integral form and use it to construct $k$ and
$m$ smoothly up to $P_S$. We then reconstruct the metric and verify
that the corresponding conformal factor $u$, defined by
$g=u^2g_{\mathbb S^n}$, extends smoothly and positively across
$P_S$.

\begin{lemma}[The south-pole coordinate system]\label{ode:pole-reduction}
Let $g$ be a radial metric of the form~\eqref{cylinder} that
extends smoothly across $P_S$ and satisfies $I_a(g)=\Lambda$
near $P_S$. On a sufficiently small punctured neighborhood of
$P_S$, the squared radius $s:=e^{2v}$ is a
valid radial coordinate. Define $k:=q/s$ and $m:=\sigma_1(A_g)$,
and regard both as functions of $s$. Write
$k_s:=\dd k/\dd s$ and $m_s:=\dd m/\dd s$.
Then $k$ is the tangential sectional curvature, and the pair
$(k,m)$ satisfies the first-order system
\begin{equation}\label{pole-system}
 \begin{split}
 s k_s&=m-\frac n2 k,\\
 2(1-sk)m_s&=-\frac\Lambda n-\frac12m^2
 +\frac{n-2}{2}km+\frac{(n-1)(a+4)}8k^2.
 \end{split}
\end{equation}
Here $s$ is the independent variable, and $k=k(s)$ and $m=m(s)$
are the two scalar unknowns.
\end{lemma}
\begin{proof}
The definitions~\eqref{variables} and~\eqref{pole:variables}
give
\begin{equation}\label{pole:chain-rule}
 s'=-2xs,\qquad \frac{\dd}{\dd t}=-2xs\frac{\dd}{\dd s}.
\end{equation}
By~\eqref{09092117}, $x>0$ near $P_S$, so $s'\ne0$ and this
change of variable is valid. The curvature formulas
\eqref{radial-curvatures} and definitions~\eqref{pole:variables}
give
\begin{equation}\label{pole:curvature-relations}
 k=K_{\rm tan},\qquad q=sk,\qquad J=sm.
\end{equation}
Differentiating $q=1-x^2$ in~\eqref{variables}, and using
\eqref{pole:chain-rule} and~\eqref{pole:curvature-relations}, yields
\begin{equation}\label{pole:q-derivative}
 k+sk_s=\frac{\dd q}{\dd s}=\frac ys
       =m-\frac{n-2}{2}k.
\end{equation}
This is the first equation in~\eqref{pole-system}.
Differentiating $J=sm$ in~\eqref{pole:curvature-relations} by
\eqref{pole:chain-rule} gives
\begin{equation}\label{pole:J-derivative}
 J'=-2xs(m+sm_s).
\end{equation}
Substitution of~\eqref{pole:curvature-relations} and
\eqref{pole:J-derivative} into~\eqref{K-J} gives
\begin{equation}\label{pole:K-identity}
 \frac{\cK}{s^2}=2x^2m_s+\frac12m^2
 -\frac{n-2}{2}km-\frac{(n-1)(a+4)}8k^2.
\end{equation}
The integration leading to~\eqref{first-integral} applies on this
pole neighborhood. Smoothness at $P_S$ and~\eqref{09092117}
make its integration constant zero, exactly as in
\eqref{09281548}. Together with~\eqref{pole:variables}, this gives
$\cK/s^2=-\Lambda/n$.
Using $x^2=1-sk$ from~\eqref{variables} and
\eqref{pole:curvature-relations} in~\eqref{pole:K-identity}
proves the second equation in~\eqref{pole-system}.
\end{proof}

Lemma~\ref{ode:pole-reduction} shows that a smooth radial metric
with $I_a(g)=\Lambda$ satisfies~\eqref{pole-system} near $P_S$.
Conversely, let us start with a smooth solution $(k,m)$ of
\eqref{pole-system} on an open interval where $s>0$ and
$1-sk(s)>0$. At this stage, $k$ and $m$ are two scalar
functions; we will construct a metric and then identify them
with its curvatures. For the south-pole branch, we must take
\begin{align}\label{10041746}
x:=\sqrt{1-sk}
\end{align}
 and recover $t$ and $v$ by
\begin{equation}\label{pole:time-reconstruction}
 \frac{\dd t}{\dd s}=-\frac{1}{2xs},\qquad
 v(t):=\frac12\log s(t).
\end{equation}
Here $s(t)$ is the local inverse of $t(s)$, which exists
because $\dd t/\dd s\ne0$. Define $g$ by~\eqref{cylinder}.
The corresponding conformal factor on the sphere is defined by
\begin{align}\label{09102328}
	u(\theta(t)):=e^{v(t)}\cosh t=\frac{1+e^{-2t}}2\sqrt{\frac{s}{e^{-2t}}}.
\end{align}
Indeed, \eqref{cylinder} and~\eqref{09092223} give
$g_{\Sph^n}=\sech^2t(\dd t^2+g_{\Sph^{n-1}})$, so
the same metric satisfies
\[
 g=e^{2v}(\dd t^2+g_{\Sph^{n-1}})=u^2g_{\Sph^n}
\]
on the corresponding region away from $P_S$.
Equation~\eqref{pole:time-reconstruction} also gives $s'=-2xs$
and $v'=-x$. At a smooth south pole, \eqref{09092117} requires
$x\to1$ as $t\to+\infty$, which determines the positive
square root in~\eqref{10041746}.

Using~\eqref{pole:chain-rule} and the first equation of
\eqref{pole-system}, the phase variables are explicitly
\begin{equation}\label{pole:phase-reconstruction}
 \begin{split}
 x&=\sqrt{1-sk},\qquad
 y=s\left(m-\frac{n-2}{2}k\right),\\
 z&=-2xs\frac{\dd}{\dd s}
       \left[s\left(m-\frac{n-2}{2}k\right)\right].
 \end{split}
\end{equation}
These formulas agree with $x=-v'$, $y=x'$, and $z=y'$ by
\eqref{pole:time-reconstruction} and~\eqref{pole:chain-rule}.
Consequently,~\eqref{variables} and~\eqref{radial-curvatures}
identify $J=sm$, $K_{\rm tan}=k$, and $\sigma_1(A_g)=m$ for the
reconstructed metric. The calculation~\eqref{pole:K-identity}
therefore applies, and the second equation in~\eqref{pole-system}
gives
\begin{equation}\label{pole:reconstructed-K}
 \cK=-\frac{\Lambda}{n}s^2
      =-\frac{\Lambda}{n}e^{4v},\qquad \cK'=-4x\cK,
\end{equation}
where the derivative follows from~\eqref{pole:chain-rule}.
Substituting~\eqref{pole:reconstructed-K} into
\eqref{09092216}, and using $x>0$, gives $z'=G$.
Equation~\eqref{radial:equivalence} then implies $I_a(g)=\Lambda$.

The preceding construction is defined for $s>0$.
The next lemma constructs $(k,m)$ smoothly up to $s=0$ and
proves that the conformal factor~\eqref{09102328} extends
smoothly and positively across $P_S$.

\Needspace*{6\baselineskip}
\begin{lemma}[Smooth south-pole reconstruction at $s=0$]\label{ode:pole}
Fix $n\ge4$, $a,\Lambda\in\R$, and $c>0$.
There exists $\varepsilon>0$ such that~\eqref{pole-system}
has a smooth solution $(k,m)$ on $[0,\varepsilon]$ satisfying
\begin{equation}\label{pole:initial-values}
 m(0)=c,\qquad k(0)=\frac{2c}{n},\qquad
 1-sk(s)>0\quad\text{for }0\le s\le\varepsilon.
\end{equation}
This solution is unique near $s=0$ among solutions for which
$k$ is bounded and $m\to c$ as $s\downarrow0$.
For every fixed parameter triple $(c,\Lambda,a)$, there is a
neighborhood in parameter space on which the same $\varepsilon$
can be used, and the corresponding solution $(k,m)$ depends
smoothly on $s$ and on the parameters.

For any finite integration constant in~\eqref{pole:time-reconstruction},
the conformal factor $u$ in~\eqref{09102328} extends smoothly and
positively across $P_S$. Hence $g=u^2g_{\Sph^n}$ extends smoothly
across $P_S$, with $I_a(g)=\Lambda$ and positive scalar curvature
near $P_S$. Here $s\downarrow0$ corresponds to $t\to+\infty$.
If the integration constant is chosen smoothly in the parameters,
so is the extended conformal factor.

Choosing the negative square root in~\eqref{pole:phase-reconstruction}
gives the analogous conclusion at $P_N$ as $t\to-\infty$,
with the same functions $k$ and $m$.
\end{lemma}
\begin{proof}
\emph{Step 1: resolve the equation for $k$ by averaging.}
Multiplying the first equation in~\eqref{pole-system} by
$s^{n/2-1}$ and integrating from $\delta$ to $s$ gives
\begin{equation}\label{pole:integrated-k}
 s^{n/2}k(s)-\delta^{n/2}k(\delta)
   =\int_\delta^s\tau^{n/2-1}m(\tau)\,\dd\tau.
\end{equation}
If $k$ is bounded, the second term on the left of
\eqref{pole:integrated-k} tends to zero as $\delta\downarrow0$.
Dividing by $s^{n/2}$ and setting $\tau=sr$ therefore yields
\begin{equation}\label{pole-average}
 k(s)=\int_0^1r^{n/2-1}m(sr)\,\dd r.
\end{equation}
In particular,~\eqref{pole-average} gives $k(0)=2c/n$ when
$m(s)\to c$.

\emph{Step 2: solve the integral equation.}
Define the bounded averaging operator
\begin{equation}\label{pole:averaging-operator}
 (\mathcal A f)(s):=\int_0^1r^{n/2-1}f(sr)\,\dd r,
 \qquad \|\mathcal A f\|_\infty\le\frac2n\|f\|_\infty.
\end{equation}
Substituting $k=\mathcal A m$ from~\eqref{pole-average} into the
second equation in~\eqref{pole-system} leads to the fixed point
operator
\begin{equation}\label{09300943}
 (\mathcal T f)(s):=c+\int_0^s
 \frac{-\Lambda/n-\tfrac12 f(\tau)^2
 +\tfrac{n-2}{2}(\mathcal A f)(\tau)f(\tau)
 +\tfrac{(n-1)(a+4)}8(\mathcal A f)(\tau)^2}
 {2\bigl(1-\tau(\mathcal A f)(\tau)\bigr)}\,\dd\tau.
\end{equation}
Work on the closed ball
\begin{equation}\label{pole:fixed-point-ball}
 \mathcal B_c:=\{f\in C([0,\varepsilon]):
                         \|f-c\|_\infty\le c/2\}.
\end{equation}
For $f\in\mathcal B_c$, equation~\eqref{pole:averaging-operator}
gives $\|\mathcal A f\|_\infty\le3c/n$. Thus
$\varepsilon\le n/(6c)$ ensures that the denominator factor
$1-\tau(\mathcal A f)(\tau)$ in~\eqref{09300943} is at least
$1/2$. The integrand and its derivatives with respect to
$f(\tau)$ and $(\mathcal A f)(\tau)$ are consequently bounded.
The mean value theorem, together with
\eqref{pole:averaging-operator}, gives constants $M,L$ independent
of small $\varepsilon$ such that
\begin{equation}\label{pole:contraction-estimates}
 \|\mathcal T f-c\|_\infty\le M\varepsilon,
 \qquad
 \|\mathcal T f-\mathcal T\widetilde f\|_\infty
 \le L\varepsilon\|f-\widetilde f\|_\infty.
\end{equation}
Choose $M\varepsilon<c/4$ and $L\varepsilon<1/2$ in
\eqref{pole:contraction-estimates}. The contraction theorem then
gives a unique fixed point $m\in\mathcal B_c$. Setting
$k=\mathcal A m$, equations~\eqref{pole-average} and
\eqref{09300943} give~\eqref{pole-system}.

For any other solution with $k$ bounded and $m(s)\to c$,
\eqref{pole-average} holds and the second equation of
\eqref{pole-system} integrates to~\eqref{09300943}. Its trace
lies in~\eqref{pole:fixed-point-ball} after shortening the interval.
The contraction estimate~\eqref{pole:contraction-estimates} proves
local uniqueness; ordinary ODE uniqueness for~\eqref{pole-system}
on $s>0$ extends equality throughout the common interval.

\emph{Step 3: regularity and parameter dependence.}
Equation~\eqref{09300943} first gives $m\in C^1$. Differentiating
\eqref{pole:averaging-operator} under the integral gives
\begin{equation}\label{pole:averaging-derivatives}
 \frac{\dd^j}{\dd s^j}(\mathcal A f)(s)
 =\int_0^1r^{n/2+j-1}f^{(j)}(sr)\,\dd r.
\end{equation}
Thus~\eqref{pole:averaging-derivatives} preserves $C^j$ regularity,
and the second equation in~\eqref{pole-system} improves the
regularity of $m$ by one derivative. Iteration proves that $k$
and $m$ are smooth up to $s=0$.

The strict inequalities in~\eqref{pole:contraction-estimates}
allow a common interval and a fixed neighborhood of the fixed
point for parameters near any given $(c,\Lambda,a)$.
On this neighborhood the denominator in~\eqref{09300943}
stays positive, the map is smooth in the parameters and in $f$,
and its derivative in $f$ has norm less than one.
The Banach-space implicit function theorem gives
smooth parameter dependence of $m$ in $C([0,\varepsilon])$.
Differentiating~\eqref{09300943} in the parameters and then using
\eqref{pole:averaging-derivatives} and~\eqref{pole-system}
bootstraps to joint smoothness in $s$ and the parameters.

By~\eqref{pole:initial-values} and~\eqref{pole:phase-reconstruction},
the variables $(x,y,z)$ are jointly smooth in $s$ and the parameters
up to $s=0$, with limit $(1,0,0)$ as $s\downarrow0$.

\emph{Step 4: smoothness in the original sphere chart.}
Use stereographic coordinates centered at the south pole,
\begin{equation}\label{pole:stereographic-coordinates}
\begin{gathered}
 \xi=\cot(\theta/2)\,\Theta=e^{-t}\Theta,\qquad
 w:=|\xi|^2=e^{-2t},\\
 (\sin\theta\,\Theta,\cos\theta)
 =\left(\frac{2\xi}{1+|\xi|^2},
              \frac{|\xi|^2-1}{1+|\xi|^2}\right).
\end{gathered}
\end{equation}
The coordinate $\xi$ is regular at $P_S$, where $\xi=0$.
The squared slice radius $s$ in~\eqref{pole:variables} differs
from the squared coordinate radius $w$. We will show that $w=sH(s)$
for a smooth positive function $H$. Inverting this relation
will express the conformal factor as a smooth function of $|\xi|^2$.

Indeed, by~\eqref{pole:initial-values},
\begin{equation}\label{pole:regular-time-correction}
 F(s):=\frac{(1-sk(s))^{-1/2}-1}{s}
 =\frac{k(s)}
 {\sqrt{1-sk(s)}\bigl(1+\sqrt{1-sk(s)}\bigr)}
\end{equation}
extends smoothly to $s=0$, jointly with the parameters. Hence
\eqref{pole:time-reconstruction} can be written as
\[
 \frac{\dd t}{\dd s}
 =-\frac1{2s}-\frac12F(s).
\]
Integrating~\eqref{pole:time-reconstruction}, using
\eqref{pole:regular-time-correction}, gives
\begin{align}\label{pole:integrated-time} 
 t(s)
 = C
 -\frac12\log s
 -\frac12\int_0^sF(\sigma)\,\dd\sigma.
 \end{align}
Here $C=t(\varepsilon)+\tfrac12\log\varepsilon
+\tfrac12\int_0^\varepsilon F(\sigma)\,\dd\sigma$.
In particular, $t\to+\infty$ as $s\downarrow0$.
By~\eqref{pole:stereographic-coordinates},
$w=e^{-2t}$, so~\eqref{pole:integrated-time} yields
\begin{equation}\label{pole:smooth-radius-map}
 w=sH(s),\qquad
 H(s):=e^{-2C}
       \exp\!\left(\int_0^sF(\sigma)\,\dd\sigma\right),
 \qquad H(0)=e^{-2C}>0.
\end{equation}
Thus the map $s\mapsto w$ has positive derivative at $s=0$.
The inverse function theorem therefore gives a smooth inverse
$s=s(w)$ near $w=0$, and
\begin{equation}\label{pole:smooth-radius-ratio}
 \frac{s(w)}{w}
 =\frac1{H(s(w))}>0,
 \qquad
 \left.\frac{s(w)}{w}\right|_{w=0}=e^{2C}.
\end{equation}

It remains only to express the conformal factor in this regular
variable. Combining~\eqref{09102328},
\eqref{pole:stereographic-coordinates}, and
\eqref{pole:smooth-radius-ratio}, we obtain
\begin{equation}\label{pole:smooth-conformal-factor}
 u(\xi)=\frac{1+|\xi|^2}{2}
        \sqrt{\frac1{H(s(|\xi|^2))}},
 \qquad
 u(0)=\frac12e^{C}>0.
\end{equation}
The right-hand side is a smooth positive function of
$|\xi|^2$. Hence $u$, and therefore
$g=u^2g_{\Sph^n}$, extends smoothly across $P_S$.

Equations~\eqref{pole:reconstructed-K} and
\eqref{radial:equivalence} give $I_a(g)=\Lambda$ away from the
pole, and hence at the pole by continuity. Moreover,
\eqref{id:notation} and~\eqref{pole:variables} give
$R_g=2(n-1)m$, so $R_g>0$ near $P_S$ by
\eqref{pole:initial-values}.

Finally, if $t(\varepsilon)$ depends smoothly on
$(c,\Lambda,a)$, then so do $C$ and $H$. The parameter-dependent
inverse function theorem applied to~\eqref{pole:smooth-radius-map}
therefore shows that $s(w)$ and the extended conformal factor
in~\eqref{pole:smooth-conformal-factor} also depend smoothly on the
parameters on a common neighborhood of the pole.

Changing the square-root sign in~\eqref{pole:phase-reconstruction}
sends $(x,y,z)$ to $(-x,y,-z)$, with limit $(-1,0,0)$ as
$s\downarrow0$. Choosing opposite integration constants
in~\eqref{pole:time-reconstruction} gives the
reflection~\eqref{radial:reflection}. The corresponding metric
is the pullback under $\theta\mapsto\pi-\theta$, so the north-pole
conclusion follows from~\eqref{pole:smooth-conformal-factor}.
\end{proof}

For later use, choose a common $\varepsilon$ for nearby parameter
values $(c,\Lambda,a)$. Choose the constant $t(\varepsilon)$
in~\eqref{pole:time-reconstruction} to depend smoothly on these parameters.
Equations~\eqref{pole:integrated-time}--\eqref{pole:smooth-conformal-factor}
then give smooth parameter dependence of the extended conformal
factor $u$.
In Section~\ref{sec:high-neck}, we will adjust $t(\varepsilon)$
to place the neck at $t=0$.

We consider $a$ near $-4$ because Section~\ref{sec:high-neck}
uses a perturbation from $a=-4$.
The same local construction works near any fixed value of $a$.
Fix the south-pole Schouten trace $\sigma_1(A_g)$ at $n/2$.
For $\Lambda$ near the round value $nB/8$ in
\eqref{radial:round-mixed-curvature}, Lemma~\ref{ode:pole} gives
local solutions on a common interval near $s=0$.
These solutions depend smoothly on $\Lambda$.
Starting at a fixed small positive value of $s$ in this interval,
we extend the corresponding trajectories backwards in $t$.
The round trajectory~\eqref{round} reaches the section
$x=0$, $y>0$ in finite time, at $(0,1,0)$.
Since $x'=y=1$ there by~\eqref{system}, the crossing is
transverse. Continuous dependence on initial data and parameters
therefore ensures that nearby trajectories also reach this section.

As $\Lambda$ varies, these intersection points form a curve.
Each point gives initial data for a metric that extends smoothly
across $P_S$ as $t\to+\infty$.
In Section~\ref{sec:high-neck}, we will vary the neck data until
the resulting trajectory meets this curve.
Reflection gives the north-pole curve, used only for
the limiting comparison in that section.
The next lemma proves smoothness of these curves and gives
uniform estimates near $P_S$.

\begin{lemma}[Initial data for the south pole]\label{ode:caps}
Let $\Sigma:=\{(x,y,z):x=0,\ y>0\}$. For $(\Lambda,a)$
sufficiently close to
\[
 \left(\frac{n^2(n-4)}8,-4\right),
\]
there exists a smooth map $P(\Lambda,a)\in\Sigma$ such that,
for each fixed $a$, the map $\Lambda\mapsto P(\Lambda,a)$
is a smooth embedding and
\[
 P(nB/8,a)=(0,1,0).
\]

Let $X(\tau;\Lambda,a)$ be the solution of~\eqref{system}
with $X(0;\Lambda,a)=P(\Lambda,a)$. It gives a metric $g$
satisfying
\[
 I_a(g)=\Lambda,\qquad \sigma_1(A_g)(P_S)=\frac n2,
\]
which extends smoothly across $P_S$ as $\tau\to+\infty$.
Moreover,
\begin{equation}\label{tail}
 X-(1,0,0)=b e^{-2\tau}(-1,2,-4)+O(e^{-4\tau}),
 \qquad \tau\to+\infty,
\end{equation}
where $b=b(\Lambda,a)>0$ is smooth and $b(nB/8,a)=2$.
The remainder estimates hold after any fixed number of
derivatives in $\tau$, $\Lambda$, and $a$, uniformly in a
sufficiently small parameter neighborhood.

Reflection under $(0,y,z)\mapsto(0,y,-z)$ gives the
corresponding north-pole curve.
\end{lemma}
\begin{proof}
By~\eqref{cylinder} and~\eqref{variables}, a constant rescaling
adds a constant to $v$ and leaves $X$ unchanged. The curvature
formulas~\eqref{radial-curvatures} and~\eqref{radial:mixed-curvature}
therefore give, for $\lambda>0$,
\begin{equation}\label{radial:constant-scaling}
 \widetilde g=\lambda g,\qquad
 \widetilde X=X,\qquad
 \sigma_1(A_{\widetilde g})=\lambda^{-1}\sigma_1(A_g),\qquad
 I_a(\widetilde g)=\lambda^{-2}I_a(g).
\end{equation}
Thus~\eqref{radial:constant-scaling} allows us to fix
\[
 c=m(0)=\frac n2.
\]
With this value of $c$ fixed, apply Lemma~\ref{ode:pole} at
\[
 (\Lambda,a)
 =
 \left(\frac{n^2(n-4)}8,-4\right).
\]
For $(\Lambda,a)$ sufficiently close to this point, the resulting
solutions are defined on a common interval $[0,\varepsilon]$ and
depend smoothly on $(s,\Lambda,a)$.
The estimates below hold uniformly for $a$ sufficiently close
to $-4$ and $\Lambda$ sufficiently close to $nB/8$.
The averaging formula~\eqref{pole-average}, together with
$m(0)=c=n/2$ in~\eqref{pole:initial-values}, gives $k(0)=2c/n=1$ for every $(\Lambda,a)$
in this family. Thus $\partial_\Lambda k(0)=0$, where the
partial derivative is taken at fixed $s$ and $a$.
By the smoothness in Lemma~\ref{ode:pole}, Taylor's theorem
at $s=0$ now gives
\begin{equation}\label{pole:cap-k-expansion}
 k=1+O(s),\qquad \partial_\Lambda k=O(s).
\end{equation}
To compare their trajectories, fix a section $x=x_0$ with
$x_0<1$ sufficiently close to $1$. By~\eqref{pole:phase-reconstruction},
this section is characterized by $sk=1-x_0^2$.
We next show that $\Lambda$ is a regular parameter for the
intersection curve. To avoid differentiating $y$ and $z$
separately, we prove that $\cK$ varies strictly monotonically
with $\Lambda$ along this family.
By~\eqref{pole:cap-k-expansion} and smoothness,
$\partial_s(sk)=k+sk_s=1+O(s)>0$, so
the implicit function theorem gives a smooth solution
$s=s(\Lambda,a)$. On the fixed section $x=x_0$, we have
\[
 s\,k(s,\Lambda,a)=1-x_0^2.
\]
Keeping $a$ and $x_0$ fixed, differentiate the section condition
$sk=1-x_0^2$ from~\eqref{pole:phase-reconstruction} with respect
to $\Lambda$:
\[
 0
 =
 \left(k+sk_s\right)\frac{\dd s}{\dd\Lambda}
 +s\,\partial_\Lambda k.
\]
Hence
\begin{equation}\label{pole:cap-section-derivative}
 \frac{\dd s}{\dd\Lambda}
 =-\frac{s\,\partial_\Lambda k}{k+sk_s}
 =O(s^2),
\end{equation}
where the last estimate follows from
\eqref{pole:cap-k-expansion} and $k+sk_s=1+O(s)$.
Differentiating~\eqref{pole:reconstructed-K} along this section
and using~\eqref{pole:cap-section-derivative} gives
\begin{equation}\label{pole:cap-parameter-regularity}
 \frac{\dd\cK}{\dd\Lambda}
 =-\frac{s^2}{n}\left(1+\frac{2\Lambda}{s}
                         \frac{\dd s}{\dd\Lambda}\right)
 =-\frac{s^2}{n}(1+O(s))<0.
\end{equation}
For fixed $a$, write
\[
 \gamma(\Lambda):=(x_0,y(\Lambda),z(\Lambda)).
\]
The map is smooth by Lemma~\ref{ode:pole} and
\eqref{pole:phase-reconstruction}. Since $\cK$ is smooth in
the phase variables by~\eqref{K},
\eqref{pole:cap-parameter-regularity} gives
\[
 D\cK(\gamma(\Lambda))[\gamma'(\Lambda)]
 =
 \frac{\dd}{\dd\Lambda}\cK(\gamma(\Lambda))<0.
\]
Hence $\gamma'(\Lambda)\ne0$. Moreover, the strict monotonicity
of $\cK\circ\gamma$ implies that $\gamma$ is injective.
The inverse function theorem applied to $\cK\circ\gamma$
recovers $\Lambda$ smoothly from $\cK$ along the image.
Thus $\gamma$ parametrizes a smooth embedded curve in
the section $x=x_0$.

Extend these trajectories backwards to $\Sigma$ by~\eqref{system}.
The round trajectory~\eqref{round} crosses both sections with
$x'=y>0$. Nearby
trajectories therefore also cross both sections.
The implicit function theorem gives smooth dependence of the
crossing times and intersection points on the parameters.
The forward flow gives the inverse
map between the sections, so the intersection points on $\Sigma$
still form a smooth embedded south-pole curve. The round trajectory gives
$(0,1,0)$ by~\eqref{round}; its pole trace is $n/2$ and
its curvature is $nB/8$ by~\eqref{radial:round-curvature-values}
and~\eqref{radial:round-mixed-curvature}.
For each parameter pair, translate time so that its intersection
with $\Sigma$ occurs at $\tau=0$.
The smooth crossing time gives a smoothly varying integration
constant $C$ in~\eqref{pole:integrated-time}. The reversibility
\eqref{radial:reflection} gives the reflected curve in the statement.

It remains to prove~\eqref{tail}. Equations~\eqref{pole:initial-values},
\eqref{pole:cap-k-expansion}, and~\eqref{pole:phase-reconstruction}
give the Taylor expansions
\begin{equation}\label{pole:phase-tail-in-s}
 x=1-\tfrac12s+O(s^2),\qquad
 y=s+O(s^2),\qquad z=-2s+O(s^2).
\end{equation}
With the integration constant corresponding to this normalized
time $\tau$, \eqref{pole:smooth-radius-ratio} gives
\begin{equation}\label{pole:tail-scale}
 s(\tau)=2b e^{-2\tau}+O(e^{-4\tau}),
 \qquad b=b(\Lambda,a)>0,
\end{equation}
where $b=\tfrac12e^{2C}$ is smooth and the remainder is $e^{-4\tau}$ times
a jointly smooth function of $e^{-2\tau}$ and $(\Lambda,a)$.
Combining~\eqref{pole:tail-scale} with~\eqref{pole:phase-tail-in-s} proves
\eqref{tail}, including the stated derivative estimates.
For the round trajectory, \eqref{09092223} and~\eqref{pole:variables}
give $s=\sech^2\tau$, hence $b=2$.
\end{proof}

For $\eta > 0$ sufficiently small, set $a=-4-\eta$ and define
\[
 \begin{aligned}
 \mathcal C_\eta^+
 &:=\text{the south-pole curve in }\Sigma\text{ constructed in Lemma~\ref{ode:caps}},\\
 \mathcal C_\eta^-
 &:=\{(0,y,-z):(0,y,z)\in\mathcal C_\eta^+\},
 \quad\text{the reflected north-pole curve}.
 \end{aligned}
\]
Here the construction fixes the south-pole Schouten trace at $n/2$
and varies $\Lambda$ near its corresponding round value $nB/8$.
The curve $\mathcal C_\eta^+$ corresponds to $P_S$ at
$t\to+\infty$, whereas $\mathcal C_\eta^-$ corresponds to $P_N$
at $t\to-\infty$ and is obtained only by reflection.
In particular, $\mathcal C_0^\pm$ refer to the case $a=-4$.

In Section~\ref{sec:high-neck}, $t=0$ is the symmetric neck,
where $x(0)=z(0)=0$ and $y(0)<0$.
We seek a trajectory satisfying
\[
 X(T_\eta)\in\mathcal C_\eta^+
 \qquad\text{for some }T_\eta>0.
\]

\begin{remark}
Lemma~\ref{ode:caps} also holds for $a$ near any fixed real number.
The round value of $\Lambda$ and the constants in the estimates
are adjusted accordingly.
\end{remark}

For $a<-4$, the next lemma gives a direct test for scalar-curvature
positivity: a positive constant $I_a(g)$ forces $R_g>0$.
\begin{lemma}[Automatic scalar positivity]\label{automatic}
Let $g=u^2g_{\Sph^n}$, where $u$ is smooth, positive, and radial.
If $a<-4$ and $I_a(g)=\Lambda>0$ is constant, then $R_g>0$.
\end{lemma}
\begin{proof}
By~\eqref{id:notation} and~\eqref{radial-curvatures},
$R_g=2(n-1)e^{-2v}J$, so we first prove $J>0$ at finite $t$.
The identity~\eqref{09281548} gives $\cK=-\Lambda e^{4v}/n<0$.
Since $a+4<0$, equation~\eqref{K-J} implies
\begin{equation}\label{radial:zero-trace-sign}
 xJ'=-\cK-\frac{(n-1)(a+4)}8q^2>0
 \qquad\text{whenever }J=0.
\end{equation}
In particular,~\eqref{radial:zero-trace-sign} makes every zero of $J$ simple.
Suppose $J<0$ on a maximal interval. At a finite left endpoint,
$J'<0$ and hence $x<0$; at a finite right endpoint,
$J'>0$ and hence $x>0$. The same signs hold at infinite endpoints
by the limits $x(t)\to\pm1$ in~\eqref{09092117}.
Thus $x$ must cross zero from negative to positive inside the
interval. But at $x=0$, equations~\eqref{variables} and
\eqref{system} give $x'=y=J-(n-2)/2<0$, which is impossible.
Hence $J\ge0$, and~\eqref{radial:zero-trace-sign} excludes finite zeros.

At either pole, $R_g\ge0$ by continuity. If it vanished there,
the pole would be a minimum of $\sigma_1(A_g)$ by~\eqref{id:notation}, so
$\Delta_g\sigma_1(A_g)\ge0$. Rotational invariance at the pole
would also give $A_g^\circ=0$, since every rotation of the tangent
space fixes $A_g$. Equation~\eqref{ode:scalar-form} would then imply
$I_a(g)\le0$, contradicting $\Lambda>0$.
\end{proof}
\section{A neck construction in every dimension}\label{sec:high-neck}

We prove Theorem~\ref{main:sharpness} by varying one initial
value at a symmetric neck. Our aim is to find a trajectory
that meets the south-pole curve $\mathcal C_\eta^+\subset\Sigma$
at a positive time. ODE uniqueness then identifies it with
a trajectory that extends smoothly to $P_S$ as $t\to+\infty$.
Reflection across the neck $t=0$ gives the north-pole half of
the metric.

By~\eqref{variables} and~\eqref{pole:chain-rule}, the slices initially
expand away from the neck, where $x=0$ and $y<0$.
We seek the first upward return to the matching section
\[
 \Sigma=\{x=0,\ y>0\}.
\]
After this crossing the matched south-pole cap contracts toward
$P_S$.

The symmetric shooting and continuity argument is motivated by
Gursky--Malchiodi~\cite[Sections~4.3--4.4]{GM12}.
They start near the round metric and vary the initial data
$(x,y,z)=(0,1-\varepsilon,0)$; see~\cite[(4.23), (4.39)]{GM12}.
Their construction covers dimension four and
$-1<\beta<-1/4$, equivalently $a<-16/3$;
see~\cite[Section~5.2]{GM12}.
Here we take $a=-4-\eta$ and let $\eta\downarrow0$ in every
dimension $n\ge4$. Near the neck, we use $\eta g_*$ as the
reference metric, where $g_*$ is the scalar-flat Schwarzschild
metric at $a=-4$ introduced below. This choice sets the initial
neck radius to $\sqrt\eta$.

The proof has three parts. Lemma~\ref{neck:variation} gives
control near the neck. Lemma~\ref{passage} follows these solutions
to their first upward return to $\Sigma$ and identifies the limiting
return points with the curve $\mathcal C_0^-$. Lemma~\ref{shooting}
uses its transverse intersection with $\mathcal C_0^+$ to choose
a parameter for which the return point lies on $\mathcal C_\eta^+$.
We then prove smooth closure, check positive scalar curvature,
and rescale the metric to obtain~\eqref{sharp:normalization}.

\subsection{The solutions near the neck}
At $a=-4$, every scalar-flat metric satisfies $I_{-4}=0$
by~\eqref{ode:scalar-form}. We choose a radial scalar-flat
metric with a symmetric neck of radius $1$.
By~\eqref{radial-curvatures}, these conditions give
\[
 v''=\frac{n-2}{2}\bigl(1-(v')^2\bigr),\qquad
 v(0)=v'(0)=0.
\]
Solving this equation gives the Riemannian Schwarzschild metric,
up to a constant scaling; see~\cite[p.~83]{BrayLee09} for its
standard Euclidean expression. In cylinder coordinates,
our normalization gives
\[
 g_*:=\cosh^{4/(n-2)}\!\left(\frac{(n-2)t}{2}\right)
       (\dd t^2+g_{\Sph^{n-1}})
 \quad\text{on }\mathbb R\times\Sph^{n-1}.
\]
We call $g_*$ the limiting neck metric. Its spherical slices
have minimum radius $1$ at $t=0$, and their radii tend to
infinity as $t\to\pm\infty$.
Here $g_{\Sph^{n-1}}$ is the unit round metric on the
$(n-1)$-sphere. For comparison, the unit round metric on
$\Sph^n$, away from its two poles, is
\[
 g_{\Sph^n}=\sech^2t(\dd t^2+g_{\Sph^{n-1}}).
\]
Writing $g_*=e^{2v_*}(\dd t^2+g_{\Sph^{n-1}})$ gives the
following solution of~\eqref{system} at $a=-4$, with initial
data $(x,y,z)=(0,-(n-2)/2,0)$:
\begin{equation}\label{neck}
 \begin{gathered}
 x_*=-\tanh\left(\frac{(n-2)t}{2}\right),\qquad
 q_*=\sech^2\left(\frac{(n-2)t}{2}\right),\\
 y_*=-\frac{n-2}{2}q_*,\qquad
 v_*=\frac2{n-2}\log\cosh\left(\frac{(n-2)t}{2}\right),\\
 s_*(t):=e^{2v_*(t)}
 =\cosh^{4/(n-2)}\!\left(\frac{(n-2)t}{2}\right).
 \end{gathered}
\end{equation}
We write $X_*:=(x_*,y_*,y_*')$ for this trajectory.
Substitution of~\eqref{neck} into~\eqref{radial-curvatures} gives
\[
 \sigma_1(A_{g_*})
 =e^{-2v_*}\left(y_*+\frac{n-2}{2}q_*\right)=0.
\]
Thus $g_*$ has zero scalar curvature. Since $J\equiv0$ along
$X_*$ and $a=-4$, \eqref{K-J} gives
$\cK(X_*(t))=0$ for all $t\in\R$.
In dimension four, $v_*=\log\cosh t$ also solves the equation
in~\cite[(5.3)]{GM12} at $\beta=0$, corresponding to $a=-4$.
This endpoint solution is our starting metric; it is not the
round solution used in~\cite[Section~4.4]{GM12}.
For $\eta>0$, we start directly with the neck scale of $\eta g_*$.
The squared radius of its slice at time $t$ is $\eta s_*(t)$.
By~\eqref{radial:constant-scaling}, this scaling leaves $X$ unchanged.
We therefore prescribe
\begin{equation}\label{data}
 \begin{gathered}
 x(0)=z(0)=0,\qquad v(0)=\frac12\log\eta,\\
 y(0)=-\sqrt{\frac{(n-2)^2}{4}-\frac{n-1}{4}\eta+2\mu\eta^2},
 \end{gathered}
\end{equation}
and vary the parameter $\mu$. Throughout the construction,
we retain the notation from Section~\ref{sec:radial}:
\begin{equation}\label{scale}
 s=e^{2v},\qquad k=\frac qs,\qquad
 m=\sigma_1(A_g)=\frac Js.
\end{equation}
Thus $s$ is the squared radius of the spherical slice, while
$k$ and $m$ are the tangential sectional curvature and Schouten
trace of $g$. In particular, $s(0)=\eta$.
By~\eqref{variables}, the data~\eqref{data} give
$v'(0)=0$ and $v''(0)=-y(0)>0$.
Thus $t=0$ is the neck center, with radius $e^{v(0)}=\sqrt\eta$.
The reflection~\eqref{radial:reflection} and ODE uniqueness give
symmetry about $t=0$, so it suffices to consider $t>0$.
Since $a=-4-\eta$, \eqref{radial:B-definition} gives
$B+4=(n-2)^2-(n-1)\eta$. Substituting~\eqref{data}
into~\eqref{K}, we obtain
\[
 \cK(0)=\frac12y(0)^2-\frac{B+4}{8}=\mu\eta^2.
\]
Since $v(0)=\frac12\log\eta$ and $\cK(0)=\mu\eta^2$,
the transport law~\eqref{transport} and the converse statement
of Lemma~\ref{ode:reduction} give, respectively,
\begin{equation}\label{neck:curvature-constant}
 \cK=\mu e^{4v}=\mu s^2,\qquad
 I_{-4-\eta}(g)=-ne^{-4v(0)}\cK(0)=-n\mu.
\end{equation}
In what follows, $\mu$ varies over a
fixed closed and bounded interval and $\eta\downarrow0$.
Constants may depend on this interval, but not on the individual
value of $\mu$. We choose the final parameter interval in the
matching argument.

The next lemma shows that $\eta^{-1}g$ converges to $g_*$
near the neck and that the Schouten trace $m$ converges to a
positive function $j_1$. Thus scalar curvature is positive
on each fixed compact time interval for small $\eta>0$.
The limit $L_n$ of $j_1$ will supply the initial datum
for extending the solution farther from the neck.

\Needspace*{8\baselineskip}
\begin{lemma}[The neck limit and its Schouten trace]\label{neck:variation}
As $\eta\to0^+$, the solutions with initial data~\eqref{data}
satisfy
\begin{equation}\label{neck:compact-convergence}
 X\longrightarrow X_*,\qquad
 v-\frac12\log\eta\longrightarrow v_*,\qquad
 m\longrightarrow j_1
\end{equation}
in $C^\infty$ on every fixed compact interval in $t$,
uniformly in $\mu$. In particular, $\eta^{-1}g\to g_*$ and
$s(t)/\eta\to s_*(t)$ on these intervals.
The function $j_1$ is even, independent of $\mu$, and satisfies
\begin{equation}\label{first-variation}
 \begin{gathered}
 \left(\cosh^2\left(\frac{(n-2)t}{2}\right)j_1'\right)'
 =\frac{n(n-1)}8\cosh^{-2n/(n-2)}\left(\frac{(n-2)t}{2}\right),\\
 j_1(0)=\frac{n-1}{4(n-2)},\qquad j_1'(0)=0.
 \end{gathered}
\end{equation}
In particular, $j_1>0$, $j_1'(t)>0$ for $t>0$, and
\begin{equation}\label{ell}
 \lim_{t\to\infty}j_1(t)=L_n
 :=\frac{n(n-1)}{2(n-2)^2}\int_0^\infty\sech^{2n/(n-2)}w\,\dd w>0.
\end{equation}
In dimension four, $L_4=1$. Consequently, for every fixed $T>0$,
the corresponding metrics
satisfy $R_g>0$ for $|t|\le T$ when $\eta>0$ is sufficiently
small, uniformly in $\mu$.
\end{lemma}
\begin{proof}
We first obtain convergence on fixed time intervals, then compute
the first nonzero contribution to the Schouten trace.
On a fixed compact time interval, the trajectory in~\eqref{neck}
is bounded. Smooth dependence of~\eqref{system} therefore gives
$X\to X_*$ uniformly for $\mu$ in the fixed compact interval.
Integrating $v'=-x$ from~\eqref{variables}, with~\eqref{data}
and~\eqref{scale}, gives
\begin{equation}\label{neck:integrated-scale}
 \begin{aligned}
 v(t)-\frac12\log\eta
 &=-\int_0^t x(\tau)\,\dd\tau\longrightarrow v_*(t),\\
 s(t)&=\eta\exp\left(-2\int_0^t x(\tau)\,\dd\tau\right).
 \end{aligned}
\end{equation}
Both $X$ and this difference depend smoothly on $\eta$ up to
$\eta=0$. Since~\eqref{neck} and~\eqref{variables} give $J=0$
at $\eta=0$, Taylor's formula shows that $J/\eta$ also extends
smoothly to $\eta=0$. Consequently, \eqref{scale} gives
\begin{equation}\label{neck:trace-variation}
 m=e^{-2(v-\frac12\log\eta)}\frac J\eta
 \longrightarrow
 j_1:=e^{-2v_*}\left.\partial_\eta J\right|_{\eta=0}.
\end{equation}
These limits hold with all $t$-derivatives, uniformly in $\mu$.
The parameter $\mu$ enters~\eqref{data} only at order
$\eta^2$, so it does not affect $j_1$.

To find the equation for $j_1$, we first derive the equation
for $m$. Since $a=-4-\eta$ and $m=\sigma_1(A_g)$,
\eqref{ode:scalar-form} and the second identity
in~\eqref{neck:curvature-constant} give
\[
 \Delta_g m
 =n\mu+\frac{\eta}{2}|A_g^\circ|_g^2+\frac{B}{2n}m^2.
\]
Using $s=e^{2v}$ from~\eqref{scale}, equations
\eqref{radial:laplacian} and~\eqref{radial-curvatures}
give, respectively,
\[
 s\Delta_g m=m''-(n-2)xm',
 \qquad
 s|A_g^\circ|_g^2=\frac{n-1}{ns}(y-q)^2.
\]
Multiplying the equation for $m$ by $s$ therefore yields
\begin{equation}\label{neck:trace-equation}
 m''-(n-2)xm'
 =\frac{(n-1)\eta}{2ns}(y-q)^2
   +s\left(n\mu+\frac{B}{2n}m^2\right).
\end{equation}
By~\eqref{neck:integrated-scale}, on the fixed time interval
$s\to0$ and $\eta/s\to e^{-2v_*}$.
Taking $\eta\to0$ in~\eqref{neck:trace-equation}, using
\eqref{neck:trace-variation} and $y_*-q_*=-nq_*/2$ from~\eqref{neck}, yields
\begin{equation}\label{neck:linear-trace-equation}
 j_1''-(n-2)x_*j_1'
 =\frac{n(n-1)}8e^{-2v_*}q_*^2.
\end{equation}
Multiplying~\eqref{neck:linear-trace-equation} by
$\cosh^2((n-2)t/2)$ and using~\eqref{neck} gives
\eqref{first-variation}. Expanding~\eqref{data} gives
\begin{equation}\label{neck:initial-expansion}
 y(0)=-\frac{n-2}2+\frac{(n-1)\eta}{4(n-2)}+O(\eta^2).
\end{equation}
By~\eqref{variables} and~\eqref{scale},
$m(0)=(y(0)+(n-2)/2)/\eta$.
Together with~\eqref{neck:initial-expansion}, this gives the
stated value of $j_1(0)$; evenness gives $j_1'(0)=0$.

Integration of~\eqref{first-variation} gives $j_1'>0$ for $t>0$
and $j_1>0$.
Integrating~\eqref{first-variation} once more and interchanging
the nonnegative integrals gives the finite limit
\begin{equation}\label{neck:trace-limit-integral}
 j_1(\infty)=j_1(0)+\frac{n(n-1)}{2(n-2)^2}
 \int_0^\infty\sech^{2n/(n-2)}w(1-\tanh w)\,\dd w.
\end{equation}
The integral containing $\tanh w$ in~\eqref{neck:trace-limit-integral}
equals $(n-2)/(2n)$; its contribution cancels the value $j_1(0)$
in~\eqref{first-variation}. This proves \eqref{ell}; for $n=4$,
$\int_0^\infty\sech^4w\,\dd w=2/3$ gives $L_4=1$.
The scalar-curvature positivity follows from
$R_g=2(n-1)m$ in~\eqref{radial-curvatures} and~\eqref{scale},
and the uniform convergence $m\to j_1>0$
on each fixed compact time interval.
\end{proof}

\begin{remark}[The role of $a<-4$]
The sign $a+4=-\eta<0$ is used in the trace equation
\eqref{neck:trace-equation}. By~\eqref{ode:scalar-form},
\eqref{radial-curvatures}, and~\eqref{scale}, its source term is
\[
 -\frac{a+4}{2}s|A_g^\circ|_g^2
 =\frac{(n-1)\eta}{2ns}(y-q)^2\ge0.
\]
On each fixed time interval, $\eta/s\to e^{-2v_*}$ by
\eqref{neck:integrated-scale}, so this term gives the strictly positive source in
\eqref{neck:linear-trace-equation}. The same sign gives
$j_1(0)=(n-1)/(4(n-2))>0$ through~\eqref{neck:initial-expansion}.
Thus~\eqref{first-variation} and~\eqref{ell} yield $j_1>0$
and $L_n>0$. This positive limit supplies the pole trace
$c=L_n$ for the round cap comparison in Lemma~\ref{passage}.
\end{remark}

\subsection{From the neck to the matching section}
The matching argument requires the point where each neck trajectory
first returns to $\Sigma$. Lemma~\ref{neck:variation} controls fixed
time intervals, on which $s/\eta\to s_*$ and hence $s\to0$.
To reach the matching section, we first use the squared radius $s$
to follow the solution to a fixed radius $\rho>0$. We take limits
using the data at this common radius, then continue the trajectories
over a bounded time interval to $\Sigma$.

Fix a sufficiently large time $L$ such that $x(L)<0$.
Starting at $t=L$, we may use $s$ as the independent variable
as long as $x<0$, since~\eqref{pole:chain-rule} gives
$s'=-2xs>0$. Substituting~\eqref{scale}
and~\eqref{neck:curvature-constant} into~\eqref{pole-system} gives
\begin{equation}\label{passage-equations}
 \begin{split}
 s k_s&=m-\frac n2 k,\\
 2(1-sk)m_s
 &=\mu-\frac12m^2+\frac{n-2}{2}km
       -\frac{n-1}{8}\eta k^2.
 \end{split}
\end{equation}
This is exactly~\eqref{pole-system}, with $\Lambda=-n\mu$
and $a=-4-\eta$; the subscripts denote $s$-derivatives.
For $s>0$ and $sk<1$, setting $\eta=0$ gives a limiting
system. Here $s$ is fixed as $\eta\to0$. This differs from
Lemma~\ref{neck:variation}, where $t$ stays in a fixed interval
and $s\to0$.

To prepare for matching, we identify a round comparison solution
of~\eqref{passage-equations} at $\eta=0$. With $L_n>0$
from~\eqref{ell}, set $m\equiv L_n$ and $k\equiv2L_n/n$.
The first equation in~\eqref{passage-equations} is satisfied,
and the second becomes
\[
 0=\mu-\frac12L_n^2+\frac{n-2}{n}L_n^2
   =\mu+\frac{n-4}{2n}L_n^2.
\]
Thus the limiting system has the constant solution
\begin{equation}\label{critical-d}
 \mu^*:=-\frac{n-4}{2n}L_n^2,\qquad
 k\equiv\frac{2L_n}{n},\qquad m\equiv L_n.
\end{equation}
These curvature values correspond to the round metric
$\frac{n}{2L_n}g_{\Sph^n}$ by~\eqref{radial:round-curvature-values}
and~\eqref{radial:constant-scaling}.

At $\mu=\mu^*$, the comparison solution~\eqref{critical-d}
meets $\Sigma$ at $(0,1,0)$. Nearby parameters give the
reflected north-pole curve $\mathcal C_0^-$, which will be compared
with the south-pole curve in Lemma~\ref{shooting}.

On each fixed time interval, \eqref{neck:compact-convergence}
and~\eqref{neck} give $x'=y\to y_*<0$.
Lemma~\ref{passage} will connect this decreasing neck branch
to an increasing pole branch. It gives convergence on every
fixed $[s,\rho]$, $s>0$, to the corresponding north-pole solution.
At $\mu=\mu^*$, \eqref{variables}, \eqref{scale}, and
\eqref{critical-d} give
$y/s=m-(n-2)k/2=2L_n/n>0$, and this inequality persists near
$\mu^*$ by smooth dependence. Thus $y>0$ on each fixed
$[s,\rho]$ for sufficiently small $\eta$.
The arrival times at these positive radii tend to infinity as
$\eta\to0$, so the change of direction escapes every fixed
time interval.

\Needspace*{8\baselineskip}
\begin{lemma}[From the neck to the matching section]\label{passage}
There are a compact interval of values of $\mu$ containing
$\mu^*$ in its interior and $\eta_0>0$ such that the following
hold for every $\mu$ in this interval and $0<\eta<\eta_0$.
\begin{enumerate}
\item The solution with initial data~\eqref{data} has a first
      upward return to $\Sigma$ at a finite time $T_\eta(\mu)>0$.
      It satisfies $x<0$ for $0<t<T_\eta(\mu)$ and $m>0$ for
      $0\le t\le T_\eta(\mu)$. The return time and return point
      depend continuously on $\mu$.
\item The return points satisfy
\begin{equation}\label{shooting:crossing-convergence}
 P_\eta(\mu):=(0,y_\eta(\mu),z_\eta(\mu))
 \longrightarrow P_0(\mu)
 \quad\hbox{uniformly in }\mu.
\end{equation}
      Here $P_0(\mu)$ is the first intersection with $\Sigma$ of
      the reflected north-pole solution of Lemma~\ref{ode:pole}
      with $a=-4$, $c=L_n$, and $\Lambda=-n\mu$. These points
      smoothly and regularly parametrize a local arc of
      $\mathcal C_0^-$, with
\begin{equation}\label{shooting:limiting-crossing}
 P_0(\mu)=(0,y_0(\mu),z_0(\mu)),\qquad
 P_0(\mu^*)=(0,1,0).
\end{equation}
\item The return times tend to infinity uniformly in $\mu$.
      The squared radius at return satisfies
\begin{equation}\label{neck:crossing-scale}
 s(T_\eta(\mu))\longrightarrow\frac{n}{2L_n}
 \quad\hbox{as }\eta\to0\hbox{ and }\mu\to\mu^*.
\end{equation}
\end{enumerate}
\end{lemma}
\begin{proof}
We first work on a compact parameter interval containing $\mu^*$
in its interior; we will shrink it in the last step.
All estimates below are uniform on this interval.
By Lemma~\ref{ode:pole} and compactness, the comparison pole
solutions are defined on a common interval near $s=0$.
Choose a small radius $\rho$ in that interval. By
\eqref{pole:initial-values} and~\eqref{pole:phase-reconstruction},
$y/s\to2L_n/n>0$ at the comparison pole, uniformly in $\mu$;
we may therefore require $y>0$ for $0<s\le\rho$ on these solutions.
We first reach $\rho$, then identify the limiting data there,
and finally continue from $\rho$ to the matching section.

\emph{\textbf{Step 1: reach the fixed radius $\rho$.}}
Starting at a large fixed time $L$, we solve
\eqref{passage-equations} with $s$ as the independent variable.
The estimates below keep the solution in a regular region of
this system until $s=\rho$ and preserve positive scalar curvature.

For any fixed $L>0$, \eqref{neck:compact-convergence},
\eqref{neck:integrated-scale}, and~\eqref{neck} give
\begin{equation}\label{passage:fixed-time-data}
 s(L)\to0,\qquad q(L)\to q_*(L)>0,\qquad
 m(s(L))\to j_1(L)
 \quad\text{as }\eta\to0.
\end{equation}
Here $m(s(L))$ is the Schouten trace at time
$t=L$. The data~\eqref{data}, the limiting trajectory~\eqref{neck},
and Lemma~\ref{neck:variation} give $y<0$ and $m>0$ on $[0,L]$
for small $\eta$, and hence $x<0$ on $(0,L]$ by~\eqref{system}.
Also, \eqref{neck} and~\eqref{ell} give $q_*(L)\to0$ and
$j_1(L)\to L_n$ as $L\to\infty$.
We will use these facts to choose $L$ and $\eta_0$ after
obtaining the estimates.

Since $x(L)<0$, continuity gives $x<0$ for a short time
after $L$. By~\eqref{pole:chain-rule}, $s'=-2xs>0$, so we can
use $s$ as the independent variable. Assume temporarily that
\begin{equation}\label{passage:bootstrap}
 x<0,\qquad q<\tfrac12,\qquad s<\rho,\qquad
 \tfrac12L_n<m<2L_n
\end{equation}
hold throughout an interval starting at $t=L$.
Our parameter choices below will make these inequalities hold
at $t=L$. We will then show that they persist until $s=\rho$.
Also, $\eta/s(L)\le1$, since~\eqref{data} and~\eqref{scale}
give $s(0)=\eta$ and $s$ increases
on $(0,L]$.
Multiplying the first equation in~\eqref{passage-equations}
by $s^{(n-2)/2}$ and using the product rule gives
\begin{equation}\label{passage:k-derivative}
 \frac{\dd}{\dd s}\bigl(s^{n/2}k(s)\bigr)
 =s^{(n-2)/2}m(s).
\end{equation}
Integrating~\eqref{passage:k-derivative} from $s(L)$ to $s$, then multiplying
by $s^{-(n-2)/2}$ and using~\eqref{scale}, gives
\begin{equation}\label{09121605}
 sk(s)
 =q(L)\left(\frac{s(L)}{s}\right)^{(n-2)/2}
 +s^{-(n-2)/2}\int_{s(L)}^{s}
       w^{(n-2)/2}m(w)\,\dd w.
\end{equation}
By~\eqref{passage:fixed-time-data} and~\eqref{passage:bootstrap},
$q(L)>0$ and $m>0$; hence~\eqref{09121605} gives $q=sk\ge0$.
Under~\eqref{passage:bootstrap}, \eqref{pole:chain-rule}
gives $s'=-2xs>0$, so $s\ge s(L)\ge\eta$ throughout
this interval. Hence $\eta/s\le1$, and using
$q=sk\ge0$ and $q<1/2$, we obtain
\begin{equation}\label{passage:eta-term}
 \eta k^2=(\eta/s)qk\le k/2.
\end{equation}
Since $\mu$ is bounded, \eqref{passage:bootstrap} and
\eqref{passage:eta-term}, applied to the second equation
in~\eqref{passage-equations}, give
\begin{equation}\label{passage:m-derivative}
 |m_s|\le C(1+k).
\end{equation}
To estimate the integral of $k$, divide~\eqref{09121605}
by $s$ and write the independent variable as $w$. This gives
\begin{equation}\label{passage:k-decomposition}
 k(w)=q(L)s(L)^{(n-2)/2}w^{-n/2}
 +w^{-n/2}\int_{s(L)}^w\tau^{(n-2)/2}m(\tau)\,\dd\tau.
\end{equation}
For $s(L)\le w\le s$, the bound on $m$ in~\eqref{passage:bootstrap} gives
\begin{equation}\label{passage:averaged-trace-bound}
 w^{-n/2}\int_{s(L)}^w\tau^{(n-2)/2}m(\tau)\,\dd\tau
 \le2L_nw^{-n/2}\int_0^w\tau^{(n-2)/2}\,\dd\tau
 =\frac{4L_n}{n}.
\end{equation}
Substituting~\eqref{passage:averaged-trace-bound} into~\eqref{passage:k-decomposition}
and integrating from $s(L)$ to $s$, we obtain
\begin{equation}\label{passage:k-integral}
 \begin{aligned}
 \int_{s(L)}^s k(w)\,\dd w
 &\le\frac{2q(L)}{n-2}
 \left[1-\left(\frac{s(L)}s\right)^{(n-2)/2}\right]
 +C\bigl(s-s(L)\bigr)\\
 &\le\frac{2q(L)}{n-2}+Cs.
 \end{aligned}
\end{equation}
Combining~\eqref{09121605}, \eqref{passage:averaged-trace-bound},
\eqref{passage:m-derivative}, and~\eqref{passage:k-integral} gives
\begin{equation}\label{passage-bounds}
 \begin{split}
 0\le sk(s)
 &\le q(L)\left(\frac{s(L)}{s}\right)^{(n-2)/2}+Cs,\\
 |m(s)-m(s(L))|
 &\le C(q(L)+s).
 \end{split}
\end{equation}
The constant $C$ may depend on the fixed parameter interval,
but is independent of $\mu$, $\eta$, and $L$.
Decrease $\rho$ so that $C\rho$ is small, then choose $L$ large so
that $q_*(L)$ and $|j_1(L)-L_n|$ are small.
Finally choose $0<\eta_0<\rho$ so that the temporary assumptions
in~\eqref{passage:bootstrap} hold at $t=L$ and the limits
\eqref{passage:fixed-time-data} make~\eqref{passage-bounds}
improve~\eqref{passage:bootstrap} to
\begin{equation}\label{passage:strict-bounds}
 q<1/4,\qquad 3L_n/4<m<5L_n/4
 \quad\text{for }0<\eta<\eta_0.
\end{equation}
These choices work for the whole parameter interval.
For each fixed $\eta>0$, \eqref{passage:strict-bounds} keeps
$k=q/s$ and $m$ bounded on $s(L)\le s<\rho$, with
$1-sk>3/4$ and $x=-\sqrt{1-q}<-1/2$.
These strict bounds close the bootstrap. ODE continuation extends
$(k,m)$ to $s=\rho$, and~\eqref{pole:phase-reconstruction}
reconstructs the original trajectory there.
Since~\eqref{pole:chain-rule} gives $s'=-2xs\ge s(L)>0$,
this radius is reached in finite time. This completes the passage
from the neck to the fixed comparison radius.

\emph{\textbf{Step 2: identify the limiting data at $\rho$.}}
We take limits on fixed $[s,\rho]$ using the data at $\rho$,
then identify the endpoint $s=0$ by the uniqueness in
Lemma~\ref{ode:pole}.
For each fixed $0<s<\rho$, the original arrival time $t(s)$
tends to infinity uniformly in $\mu$ by
\eqref{neck:compact-convergence} and~\eqref{neck:integrated-scale}.
The relative travel time nevertheless stays bounded:
\eqref{pole:chain-rule} and~\eqref{passage:strict-bounds} give
\begin{equation}\label{passage:relative-time}
 \frac12\log\frac\rho s\le t(\rho)-t(s)
 =\int_s^\rho\frac{\dd w}{2w\sqrt{1-wk(w)}}
 \le\frac1{\sqrt3}\log\frac\rho s.
\end{equation}

Take any sequence $\eta\to0$ with convergent parameters $\mu$.
By~\eqref{passage:strict-bounds}, the data at $\rho$ have a
convergent subsequence. On each fixed $[s,\rho]$, we have
$0\le k(w)\le1/(4s)$, bounded $m$, and $1-wk(w)\ge3/4$.
These bounds keep~\eqref{passage-equations} in a compact
regular region. ODE continuation and smooth dependence,
applied backwards from $\rho$, therefore give smooth convergence
on every such interval along the same subsequence, to a solution
of~\eqref{passage-equations} with $\eta=0$.
We retain $(k,m)$ and $\mu$ for this limit.
For every sufficiently large fixed $L$ and fixed $s>0$,
the initial term in~\eqref{passage-bounds}
vanishes as $\eta\to0$. Hence~\eqref{passage:fixed-time-data}
and~\eqref{passage-bounds} give
\begin{equation}\label{passage:limit-trace}
 0\le k(s)\le C,\qquad
 |m(s)-j_1(L)|\le C\bigl(q_*(L)+s\bigr),\qquad 0<s\le\rho.
\end{equation}
The constants $C,\rho$ are independent of $L$; increasing $L$
only decreases the allowed upper bound on $\eta$.
Letting $L\to\infty$ in~\eqref{passage:limit-trace}, using
\eqref{neck} and~\eqref{ell}, gives the bound for $m$ below.
Since $k$ is bounded, integrating the first equation in
\eqref{passage-equations} from zero gives~\eqref{pole-average},
and hence the limit of $k$:
\begin{equation}\label{passage:limiting-pole-data}
 |m(s)-L_n|\le Cs,\qquad
 \lim_{s\downarrow0}(k(s),m(s))=\left(\frac{2L_n}{n},L_n\right).
\end{equation}
Lemma~\ref{ode:pole} now identifies the limit with its regular
pole solution for $c=L_n$, $a=-4$, and $\Lambda=-n\mu$;
ODE uniqueness extends this identification to $\rho$.
Choosing the negative square root in~\eqref{pole:phase-reconstruction}
gives the north-pole branch, with
\begin{equation}\label{passage:pole-phase-endpoint}
 x(s)=-1+\frac{L_n}{n}s+O(s^2),\qquad
 \frac{y(s)}s\longrightarrow\frac{2L_n}{n}>0.
\end{equation}
For the limiting solution, \eqref{passage:relative-time} gives
$t(s)-t(\rho)\to-\infty$ as $s\downarrow0$, consistently
with its north-pole endpoint $x\to-1$.
The convergence above holds on fixed intervals
$[\epsilon,\rho]$ with $\epsilon>0$, not at the moving
neck point $s=\eta$. Thus the neck condition $x=0$
does not pass to the limiting endpoint $s=0$.

Every possible subsequential limit has thus been identified
with the pole solution at the limiting parameter.
The smooth parameter dependence in Lemma~\ref{ode:pole} and
a subsequence contradiction give convergence uniformly in $\mu$
on each fixed $[s,\rho]$.
Equations~\eqref{passage-equations} and
\eqref{pole:phase-reconstruction} then give uniform convergence
of the phase data $X$ at $\rho$.

\Needspace*{4\baselineskip}

\emph{\textbf{Step 3: continue to the matching section.}}
We now follow the convergent data to a transverse crossing.
For the round comparison~\eqref{critical-d}, choose a fixed
time segment from $s=\rho$ to slightly beyond its crossing
of $\Sigma$. On this segment, $y,m>0$, and $x$ has opposite
signs at the endpoints. Step 2 and continuous dependence for
\eqref{system} and~\eqref{pole:chain-rule} preserve these
inequalities for $\mu$ near $\mu^*$ and small $\eta$.
Since $x'=y>0$, each trajectory crosses $\Sigma$ exactly once
on this segment. The earlier portion has $x<0$ by Step 1
for the neck solution and by~\eqref{pole:phase-reconstruction}
for the comparison solution, so no earlier intersection occurs.
Step 1 also gives $m>0$ up to $\rho$, proving the positivity in (1).
The arrivals at $s=\rho$ and at $\Sigma$ are transverse, since
$s'=-2x\rho>0$ and $x'=y>0$ there, respectively, by
\eqref{pole:chain-rule} and~\eqref{system}.
Their times and points therefore depend continuously on $\mu$;
the comparison points depend smoothly on $\mu$.
Transversality also turns the uniform convergence from Step 2
into~\eqref{shooting:crossing-convergence}.

Rescaling the comparison metric by $2L_n/n$ fixes its
pole trace at $n/2$, leaves $X$ unchanged, and changes
$\Lambda$ to $-n^3\mu/(4L_n^2)$ by
\eqref{radial:constant-scaling}.
This affine dependence has nonzero derivative. Thus
Lemma~\ref{ode:caps} and~\eqref{critical-d} identify
$P_0$ as a regular parametrization of $\mathcal C_0^-$
through $(0,1,0)$, proving (2).

Finally, \eqref{neck:compact-convergence} and
\eqref{neck:integrated-scale} give $s(t)\to0$ uniformly
in $\mu$ on fixed time intervals, so
$T_\eta(\mu)>t(\rho)\to\infty$ uniformly.
The convergence of $(X,s)$ at the crossing gives
\eqref{neck:crossing-scale}, since the round crossing
satisfies $1=q=2L_ns/n$ by~\eqref{critical-d}
and~\eqref{scale}. This proves (3).
\end{proof}

\Needspace*{6\baselineskip}
\subsection{Matching the north- and south-pole curves}
Lemma~\ref{passage} gives the first-return points $P_\eta(\mu)$
and their limiting curve $\mathcal C_0^-$. We must choose $\mu$
so that $P_\eta(\mu)$ lies on the south-pole curve
$\mathcal C_\eta^+$. 
\begin{lemma}[Transverse shooting]\label{shooting}
The curves $\mathcal C_0^-$ and $\mathcal C_0^+$ meet transversely
at $(0,1,0)$. For every sufficiently small $\eta>0$, there is a
parameter $\mu_\eta$ in the interval of Lemma~\ref{passage} such that
$P_\eta(\mu_\eta)\in\mathcal C_\eta^+$.
Moreover $\mu_\eta\to\mu^*$, with $\mu^*$ given by~\eqref{critical-d}.
\end{lemma}
\begin{proof}
We use the ODE to prove that the two limiting curves cross
transversely. A short crossing arc then retains an intersection
under the perturbation in Lemma~\ref{passage}.

Fix $a=-4$ and the south-pole trace at $n/2$, and let
$X(\tau;\Lambda)$ be the family in Lemma~\ref{ode:caps},
with $x(0;\Lambda)=0$. At $\Lambda=n^2(n-4)/8$, this is the
round trajectory~\eqref{round}, with $x=\tanh\tau$ and
$y=q=\sech^2\tau$.
All parameter derivatives below are evaluated at this round value.
To determine the slope of $\mathcal C_0^+$, we vary $y-q$,
which vanishes identically along the round trajectory.
Since $x(0;\Lambda)=0$, this variation and its
$\tau$-derivative at $\tau=0$ give the two components
of the tangent vector in the $(y,z)$-plane.
Set
\begin{equation}\label{shooting:psi}
 \psi(\tau):=\left.\partial_\Lambda(y-q)(\tau;\Lambda)
              \right|_{\Lambda=n^2(n-4)/8}.
\end{equation}
Differentiating~\eqref{system} and~\eqref{variables} along
\eqref{round}, with $B=n(n-4)$ from~\eqref{radial:B-definition}, gives
\begin{equation}\label{round:transversality}
 \psi''+(6-n)\tanh\tau\,\psi'
 =(n-2)\left(4\tanh^2\tau
                 +\frac{n+2}{2}\sech^2\tau\right)\psi,
\end{equation}
where primes denote $\tau$-derivatives.
The parameter estimates~\eqref{tail} give
$\psi,\psi'=O(e^{-2\tau})$.
Multiplying~\eqref{round:transversality} by
$\cosh^{6-n}\tau\,\psi$ and integrating by parts therefore yields
\begin{equation}\label{shooting:energy}
 -\psi(0)\psi'(0)
 =\int_0^\infty\cosh^{6-n}\tau\Bigl[(\psi')^2
 +(n-2)\Bigl(4\tanh^2\tau
 +\frac{n+2}{2}\sech^2\tau\Bigr)\psi^2\Bigr]\,\dd\tau.
\end{equation}
The boundary term and integrand in~\eqref{shooting:energy}
are $O(e^{(2-n)\tau})$ at infinity, so the boundary term
vanishes and the integral is finite for $n\ge4$.
Since $x(0;\Lambda)=0$ for every $\Lambda$,
\eqref{shooting:psi}, \eqref{variables}, and~\eqref{system} give
\[
 \psi(0)=\partial_\Lambda y(0),\qquad
 \psi'(0)=\partial_\Lambda z(0).
\]
These two values are not both zero, because $\Lambda$
regularly parametrizes $\mathcal C_0^+$ by
Lemma~\ref{ode:caps}. Thus $\psi\not\equiv0$, and the
integral in~\eqref{shooting:energy} is strictly positive.
Consequently, \eqref{shooting:energy} gives
\[
 \partial_\Lambda y(0)\,\partial_\Lambda z(0)<0,
 \qquad
 \left.\frac{\dd z}{\dd y}\right|_{\mathcal C_0^+}
 =\frac{\partial_\Lambda z(0)}{\partial_\Lambda y(0)}<0.
\]
Hence $\mathcal C_0^+$ has a finite negative slope at
$(0,1,0)$. Reflection~\eqref{radial:reflection} gives a
positive slope for $\mathcal C_0^-$, proving transversality.

To obtain a match, work in a fixed neighborhood of $(0,1,0)$
where $\mathcal C_\eta^+$ is a graph over $y$ for all small $\eta$.
Such a common neighborhood exists by the nonzero $y$-component
of its tangent and the smooth dependence in Lemma~\ref{ode:caps}.
Choose a fixed closed parameter interval around $\mu^*$ so
that the arc $P_0(\mu)$ lies compactly in this neighborhood,
meets $\mathcal C_0^+$ only at $(0,1,0)$, and has endpoints
on opposite sides. By~\eqref{shooting:crossing-convergence},
the entire perturbed arc stays in this neighborhood, and its
endpoints remain on opposite sides of $\mathcal C_\eta^+$.
Along this continuous arc, the signed vertical difference
from the graph changes sign. The intermediate value theorem
therefore gives $P_\eta(\mu_\eta)\in\mathcal C_\eta^+$.
Uniform convergence and the smooth dependence of the graphs
make every limit point of $\mu_\eta$ a limiting intersection.
Its uniqueness on the fixed compact interval gives
$\mu_\eta\to\mu^*$.
\end{proof}

\begin{proof}[Completion of Theorem~\ref{main:sharpness}]
Let
\[
 \widehat g_\eta=s_\eta(t)(\dd t^2+g_{\Sph^{n-1}}),
 \qquad s_\eta(0)=\eta,
\]
be the neck metric selected by Lemma~\ref{shooting}, and set
$T_\eta:=T_\eta(\mu_\eta)$.

Keeping the neck coordinate $t$, translate the matching
south-pole solution so that its crossing of $\Sigma$ occurs
at $t=T_\eta$, and rescale its metric to have squared radius
$s_\eta(T_\eta)$ there.
The values of $(v,x,y,z)$ then agree at the match.
Uniqueness for~\eqref{system}, together with $v'=-x$
from~\eqref{variables},
identifies their continuations.
Lemma~\ref{ode:pole} and reflection~\eqref{radial:reflection}
therefore give a smooth metric on $\Sph^n$, conformal to
$g_{\Sph^n}$ and with the asserted symmetries.

By Lemmas~\ref{shooting} and~\ref{passage},
$P_\eta(\mu_\eta)\to(0,1,0)$.
The pole caps with south-pole Schouten trace $n/2$ therefore
converge smoothly, including at the pole, to the unit round cap
by Lemmas~\ref{ode:caps} and~\ref{ode:pole}.
They have positive scalar curvature for small $\eta$, and the
positive rescaling used in matching preserves this sign.
Together with Lemma~\ref{passage} and~\eqref{radial:reflection}, this gives
$R_{\widehat g_\eta}>0$ everywhere.
Since~\eqref{data} and~\eqref{variables} give $y(0)<0<q(0)=1$,
\eqref{radial-curvatures}
also shows that $\widehat g_\eta$ is not Einstein.

It remains only to impose~\eqref{sharp:normalization}.
By~\eqref{neck:curvature-constant},
$I_{-4-\eta}(\widehat g_\eta)=-n\mu_\eta$.
For $n\ge5$, Lemma~\ref{shooting} and~\eqref{critical-d}
give $\mu_\eta\to\mu^*<0$.
For $n=4$, equations~\eqref{intro:schouten-defs} and~\eqref{intro:Q}
give $I_{-4-\eta}(g)=-\Delta_g\sigma_1(A_g)-\eta\sigma_2(A_g)$.
The conformal invariance of $\int\sigma_2(A_g)\,\dd v_g$
\cite[(0.2) and the following paragraph]{CGY02}, together with
\eqref{radial:round-curvature-values}, gives its value $4\pi^2$
in our normalization. Integration therefore yields
\begin{equation}\label{four-total}
 \mu_\eta\Vol_{\widehat g_\eta}(\Sph^4)=\pi^2\eta>0.
\end{equation}
Thus $I_{-4-\eta}(\widehat g_\eta)$ and
$I_{-4-\eta}(g_{\Sph^n})$ have the same nonzero sign
by~\eqref{radial:round-mixed-curvature}.
The rescaling
\begin{equation}\label{neck:normalization}
 g_{n,\eta}:=
 \left(
 \frac{I_{-4-\eta}(\widehat g_\eta)}
      {I_{-4-\eta}(g_{\Sph^n})}
 \right)^{1/2}\widehat g_\eta
\end{equation}
therefore gives~\eqref{sharp:normalization} by
\eqref{radial:constant-scaling}, preserving all the
other asserted properties.
\end{proof}

\section*{Acknowledgements}
The author used ChatGPT and Codex (OpenAI) primarily to assist with the
proof of Theorem~\ref{main:sharpness}, notably in identifying the need for
a neck construction different from the approach of
Gursky--Malchiodi~\cite{GM12}. These tools also assisted with literature
searches, mathematical calculations, and English language editing.
The author has independently verified all mathematical arguments and
takes full responsibility for the content of this manuscript.

\end{document}